\documentclass[onefignum,onetabnum]{siamart171218}

\usepackage{lipsum}
\usepackage{amsfonts}
\usepackage{graphicx}
\usepackage{epstopdf}
\usepackage{algorithmic}
\usepackage{tikz}

\ifpdf
  \DeclareGraphicsExtensions{.eps,.pdf,.png,.jpg}
\else
  \DeclareGraphicsExtensions{.eps}
\fi

\newsiamremark{remark}{Remark}
\newsiamremark{hypothesis}{Hypothesis}
\crefname{hypothesis}{Hypothesis}{Hypotheses}
\newsiamthm{claim}{Claim}

\headers{Optimal polynomial approximation of circular arcs}{A. Vavpeti\v{c}, and E. \v{Z}agar}

\title{On optimal polynomial geometric interpolation of circular arcs according to the Hausdorff distance
\thanks{Submitted to the editors 2020.
\funding{The first author was supported by the Slovenian Research Agency
program P1-0292 and the grants J1-8131, N1-0064, and N1-0083. The second author was supported in part by the program
P1-0288 and the grant J1-9104 by the same agency.}}}

\author{Ale\v{s} Vavpeti\v{c}\thanks{Faculty of Mathematics and Physics, University of Ljubljana, Slovenia and
Institute of Mathematics, Physics and Mechanics, Lju\-blja\-na, Slovenia
  (\email{ales.vavpetic@fmf.uni-lj.si},\email{emil.zagar@fmf.uni-lj.si}).}
\and Emil \v{Z}agar\footnotemark[2]}

\usepackage{amsopn}

\def\bfm#1{\boldsymbol{#1}} 
\def\RR{\mathbb{R}}

\def\cc{c}
\def\ss{s}

\ifpdf
\hypersetup{
  pdftitle={On optimal polynomial geometric interpolation of circular arcs according to the Hausdorff distance},
  pdfauthor={Ale\v{s} Vavpeti\v{c}, and Emil \v{Z}agar}
}
\fi

\begin{document}

\maketitle

\begin{abstract}
  The problem of the optimal approximation of circular arcs by parametric polynomial curves is considered.
  The optimality relates to the Hausdorff distance and have not been studied yet in the literature.
  Parametric polynomial curves of low degree are used and a geometric continuity is prescribed at the boundary
  points of the circular arc. A general theory about the existence and the uniqueness of the optimal approximant
  is presented and a rigorous analysis is done for some special cases for which the degree of the polynomial curve
  and the order of the geometric smoothness differ by two. This includes practically interesting cases of
  parabolic $G^0$, cubic $G^1$, quartic $G^2$ and quintic $G^3$ interpolation. Several numerical examples are
  presented which confirm theoretical results.
\end{abstract}

\begin{keywords}
  geometric interpolation, circular arc, best approximation, Hausdorff distance
\end{keywords}

\begin{AMS}
65D05, 65D17, 41A05, 41A50
\end{AMS}

\section{Introduction}\label{sec:intro}
An efficient representation of circular arcs is an important practical as well as theoretical issue.
Since they are fundamental geometric objects they have been widely used in practical applications,
such as geometric design, geometric modelling and computer aided manufacturing. On the other hand, they
have been studied already in ancient history and they are still attracted by present researchers.

It is well known that a circular arc does not possess a polynomial parametric representation.
Although it can be exactly represented in a parametric rational form, it is still an interesting
question how close to a circular arc a parametric polynomial can be. The answer definitely depends on
the measure of the error. In this paper we shall study the well known Hausdorff
distance of certain classes of parametric polynomial curves to a circular arc.
The geometric parametric polynomial interpolants of a circular arc will be considered as approximants and the best one
will be determined according to the Hausdorff distance. This will provide geometric polynomial curves which
can be as close as possible to the circular arcs if both are considered as sets of points in $\RR^2$.
Although a lot of research has been done on optimal approximation of circular arcs by various parametric polynomial curves,
the optimality of the approximant has not been studied yet according to the well known Hausdorff distance.
It seems that the reason is that the problem becomes much more difficult as if some other measure of the error is considered, such as the simplified radial error, e.g., which will be formally defined later.

The paper is organized as follows. In \cref{sec:prelim} we recall some preliminaries which will be later used in the
process of the analysis and the construction of the best approximants. In the next section we consider the best
approximation of a circular arc by parabolic curves which only interpolate the boundary points of the circular arc.
The problem was already
studied in \cite{Morken-91-circles} but according to a simplified distance. In \cref{sec:general_error} the simplified version of the
error is studied since the results can later be used also for the Hausdorff  distance. In order to demonstrate the new approach,
the revision of the parabolic approximation is revisited in \cref{sec:parabolic_case_revisited}.
In the next three sections the cubic, the quartic and the quintic cases are studied in detail. The paper is
concluded by \cref{sec:conclusion} with some final remarks.

\section{Preliminaries}\label{sec:prelim}
Let us denote a considered circular arc by $\bfm{c}$. More precisely, $\bfm{c}$
will be parametrized as $\bfm{c}\colon[-\varphi,\varphi]\to\RR^2$, $0<\varphi\leq \pi/2$.
Since it is enough to consider the unit circular arc centered at the origin of a particular
coordinate system and symmetric with respect to the first coordinate axis, we can assume that
$\bfm{c}(\alpha)=(\cos\alpha,\sin\alpha)^T$. A polynomial approximation of $\bfm{c}$ will be
denoted by $\bfm{p}_n\colon[-1,1]\to\RR^2$, where $\bfm{p}_n=(x_n,y_n)^T$ and $x_n$, $y_n$
are scalar polynomials of degree at most $n$. In order to simplify the notation, we
will write $\cc:=\cos\varphi$ and $\ss:=\sin\varphi$. Note that $0\leq c<1$.
The Bernstein-B\'ezier representation
of $\bfm{p}_n$ will be considered, i.e.,
\begin{equation}\label{p_Bern_Bez_form}
  \bfm{p}_n(t)=\sum_{j=0}^n B_j^n(t)\,\bfm{b}_j,
\end{equation}
where $B_j^n$, $j=0,1,\dots,n$, are (reparameterized) Bernstein polynomials over $[-1,1]$,
given as
\begin{equation*}
  B_j^n(t)=\binom{n}{j}\left(\frac{1+t}{2}\right)^{j}\left(\frac{1-t}{2}\right)^{n-j},
\end{equation*}
and $\bfm{b}_j\in\RR^2$, $j=0,1,\dots,n$, are the control points. Note that since $\bfm{c}$ is symmetric
with respect to the first coordinate axis, also the best parametric polynomial interpolant must
posses the same symmetry, i.e., $\bfm{b}_j=r(\bfm{b}_{n-j})$, $j=0,1,\dots,n$,
where $r\colon\RR^2\to\RR^2$ is the reflection
over the first coordinate axis.
Let ${\mathcal G}_n^k$ denote a class of symmetric $G^k$ polynomial interpolants of $\bfm{c}$
of degree at most $n$.
More precisely, ${\mathcal G}_n^k$ consists of those parametric polynomials of degree at most $n$,
which geometrically interpolate boundary points of $\bfm{c}$ with an order $k$.

The radial error $\phi_{n,k}$ and a simplified radial error
$\psi_{n,k}$ of an interpolant $\bfm{p}_n\in{\mathcal G}_n^k$ will be defined as
\begin{align}
  \phi_{n,k}(t)&=\sqrt{x_n^2(t)+y_n^2(t)}-1=\|\bfm{p}_n(t)\|_2-1,\label{def:radial}\\
  \psi_{n,k}(t)&=x_n^2(t)+y_n^2(t)-1=\|\bfm{p}_n(t)\|_2^2-1,\quad t\in[-1,1],\label{def:simp_radial}
\end{align}
where $\|\cdot\|_2$ is the euclidean norm. Note that
\begin{equation}\label{rel:psi_widetildepsi}
  \phi_{n,k}=\sqrt{\psi_{n,k}+1}-1,
\end{equation}
and also note that both errors depend on the angle $\varphi$ and some parameters arising from the
$G_n^k$ interpolation.
It was shown in \cite{Eisele-Chebyshev-94} that the optimal $\bfm{p}_n$
implies minimal radial or simplified
radial error \eqref{def:radial} or \eqref{def:simp_radial} if and only if
$\phi_{n,k}$ or $\psi_{n,k}$ alternates $2n-2k-1$, i.e., the error must have $2n-2k-1$ extrema
with the same absolute value. The existence
of the optimal curve $\bfm{p}_n$ is guaranteed by the continuity and compactness argument.
However, the uniqueness and its construction are much more challenging issues.
Results for some particular cases in the case of
simplified radial error \eqref{def:simp_radial} where obtained in the pioneering paper
\cite{Dokken-Daehlen-Lyche-Morken-90-CAGD} but no optimality was proved. The parabolic approximation via
the $G^0$ parabolic curves has been presented in \cite{Morken-91-circles}. Later on several authors
considered plenty of particular cases of low degree geometric interpolants of the circular arcs.
In \cite{Goldapp-91-CAGD-circle-cubic}, the author has studied the optimality of the cubic parametric approximation
of low order geometric smoothness, but no rigorous proofs of the optimality have been provided.
In \cite{Ahn-Kim-CAD-quartics-2007},  the authors have studied the same problem and they have been able to prove
the optimality of the solution for the cubic $G^1$ and quartic $G^2$ cases. Although they claim that they obtained
optimal approximant according to the the Hausdorff distance,
this is not the case, since they have used its simplification. Particular cases of approximation of circular arcs by quartic B\'{e}zier
curves can be found in \cite{HurKim-2011} in \cite{Xiaoming_Licai-2010-JCADCG} and in \cite{Kovac-Zagar-quartics-2014}.
Quintic polynomial approximants of various geometric smoothness have been studied in
 \cite{Fang-98-CAGD-circle-quintic}. The methods how to approximate the whole circle can be found in
\cite{Jaklic-circle-CAGD-2016}  and  in \cite{Jaklic-Kozak-best-circle-2017}. In the latter paper the Hausdorff distance has been
considered but only for the case where interpolation of the boundary points is not required. A general framework for
the approximation of circular arcs by parametric polynomials was given in \cite{Vavpetic-Zagar-general-circle-19}
and some particular cases following this approach can be found in \cite{Vavpetic-CAGD20}.
The optimal approximants of maximal geometric smoothness are characterized in \cite{Knez-Zagar-max-geometric-2018}.\\
All the above studies involve the simplified radial error \eqref{def:simp_radial} and it seems that no results are available in
the literature on optimal approximation with respect to the the radial error \eqref{def:radial}.
However, it is important to study the existence and the uniqueness of the optimal approximants with respect to the radial error,
since it was  was shown in \cite{Jaklic-Kozak-best-circle-2017} that the radial error induces the Hausdorff
distance in this case.\\
In order to see that things are much more complicated if the radial error is considered,
let us take a quick look at the parabolic case first.

\section{Parabolic $G^0$ case}\label{sec:parabolicG0direct}

Only three control points have to be considered here. Since we are looking for a $G^0$ interpolant, we must have
\begin{equation*}
  \bfm{b}_0=(\cc,-\ss)^T,\quad
  \bfm{b}_1=(d,0)^T,\quad
  \bfm{b}_2=(\cc,\ss)^T
\end{equation*}
where $d\in\RR$ is an unknown parameter.
Note that due to the convex hull property we obviously have $d>1$.
Our goal is to find a parameter $d$ such that
$\phi_2:=\phi_{2,0}$ alternates three times with minimal absolute
extremal value.
It is easy to verify that the extrema of $\phi_2$ occur at
\begin{equation*}
  t_0=0,\quad t_1=\frac{\sqrt{d^2+\cc^2-2}}{d-\cc},\text{ and } t_2=-t_1.
\end{equation*}
Since we require $0<t_1<1$ the unknown parameter $d$ is bounded by $\sqrt{2-c^2}<d<\tfrac{1}{c}$.
The function $\phi_2$ is an even function and the alternation condition can be expressed as
$\phi_{2}(0)+\phi_{2}(t_1)=0$. This leads to
\begin{equation}\label{eq:f(d)}
f(d):=d^4-8d^3+2(c^2+4c+6)d^2-8c(4-c)d+c^4-8c^3+12c^2+4=0.
\end{equation}
During the process of determining $f$ some square roots are involved and we observe that $d<4-c$.
Thus
$d$ must be on the interval $I_d:=[\sqrt{2-c^2},\inf\{4-c,\tfrac{1}{c}\}]$.
Since $\phi_{2}(0)=\tfrac{d-(2-c)}{2}$ is the maximal value of $\phi_2$ on $[-1,1]$, we must find
a zero of $f$ in $I_d$ which is closest to $2-c>\sqrt{2-c^2}$.
The following lemma reveals its location.
\begin{lemma}\label{lem:fG0}
Let $f$ be as in  \eqref{eq:f(d)}. Then (see \cref{fig:quadraticfboth}):
\begin{enumerate}
\item $f(3-3c+c^2)<0$,
\item $f(d)>0$, for all $d\in (-\infty,2-c]$,
\item $f'(d)<0$, for all $d\in [2-c,3-3c+c^2]$.
\end{enumerate}
\end{lemma}

\begin{proof}
We have $f(3-3c+c^2)=-(1-c)^4(22(1-c^2)+(1-c^4)+8c(1+c^2))<0$.
The inequality $f(d)>0$ for all $d\in (-\infty,2-c]$ is equivalent to the inequality
$g(x)>0$ for all $x\in [0,\infty)$, where $g(x)=f(-x-c+2)$. The later inequality holds because
\begin{align*}
g(x)&=4 (1-c)^4+x \left(8 (1-c)^2 c+4 (2-2 c-x)^2+x \left((2-2 c-x)^2+4 c^2\right)\right),
\end{align*}
which is obviously positive for $x\in[0,\infty)$ and $c\in [0,1)$.\\
The inequality $f'(d)<0$ for all $d\in [2-c,3-3c+c^2]$ is equivalent to the inequality
$h(x)<0$ for all $x\in [0,1]$, where $h(x)=f'((1-2c+c^2)x+2-c)$. The later inequality holds because
\begin{align*}
h(x)&=-4(1-c)^2\left(3 x^2 (1-c)^2 c+4 \left(1-x c^2\right)+\left(4-x^2 (1-c)^3\right) x (1-c)+2(c+x)\right),
\end{align*}
which is obviously negative for $x\in[0,1]$ since $c\in [0,1)$. This concludes the proof of the lemma.
\end{proof}

Recall that our goal is to find such $d\in I_d$ which is a solution of $f(d)=0$ and as close to $2-c$ as
possible. By \cref{lem:fG0}, all solutions must be greater than $2-c$ but precisely one on
$(2-c,c^2-3c+3)\subset I_d$ which implies the best quadratic $G^0$ interpolant.
Unfortunately, its closed form is to complicated to be written here and it is
more efficient to find it numerically for a given central angle $2\varphi$.
The graphs of the radial error and
the simplified radial error of the optimal parabolic $G^0$ interpolants according to the radial and the simplified radial error
for $\varphi=\pi/2$ can be observed in \cref{fig:G20}. In \cref{tab:G20}
numerical values of optimal parameters $d_r$ according to the radial error and optimal parameters
$d_s$ according to the simplified radial error \cite{Morken-91-circles}
with corresponding values of $\phi_{2,0}$ are shown for several values of an inner angle $2\varphi$.
\begin{figure}[!htb]
\minipage{0.49\textwidth}
\centering
\includegraphics[width=0.75\linewidth]{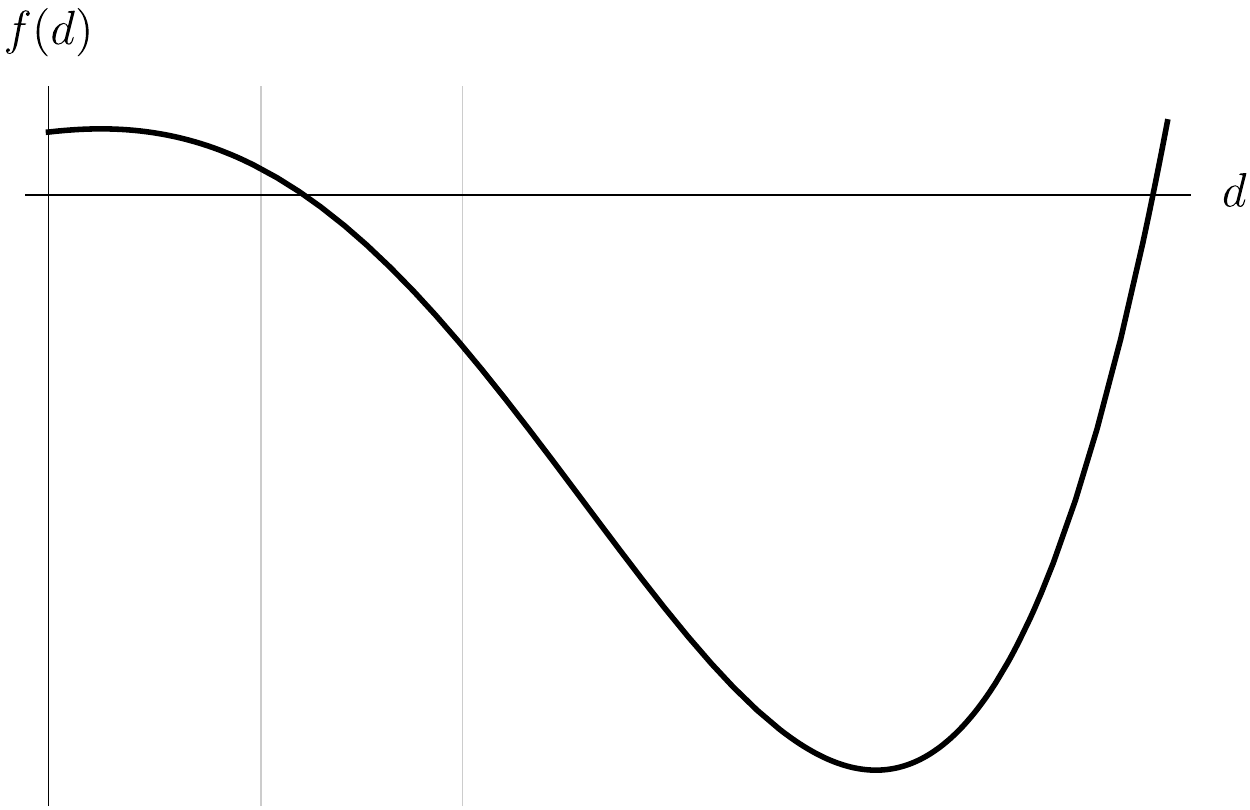}
\endminipage\hfill
\minipage{0.49\textwidth}
\centering
\includegraphics[width=0.75\linewidth]{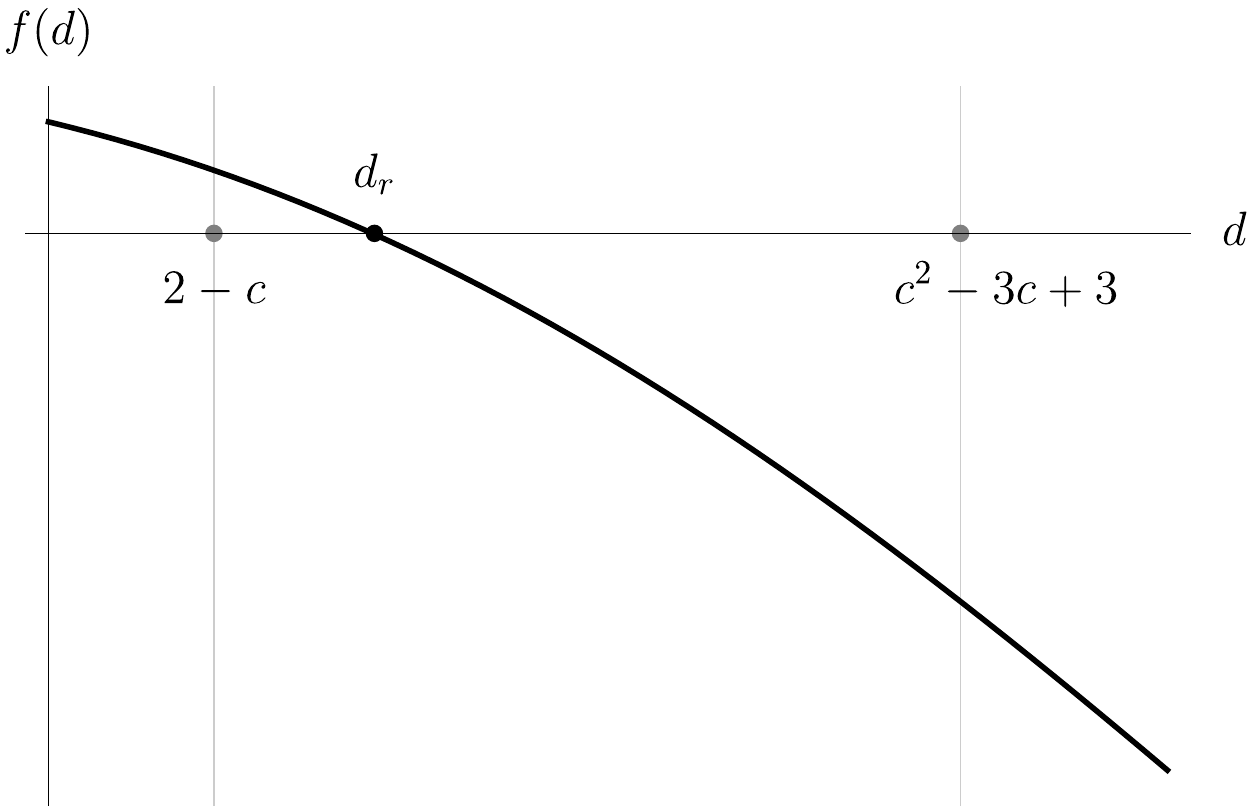}
\endminipage\hfill
\minipage{0.49\textwidth}
\centering
\includegraphics[width=0.75\linewidth]{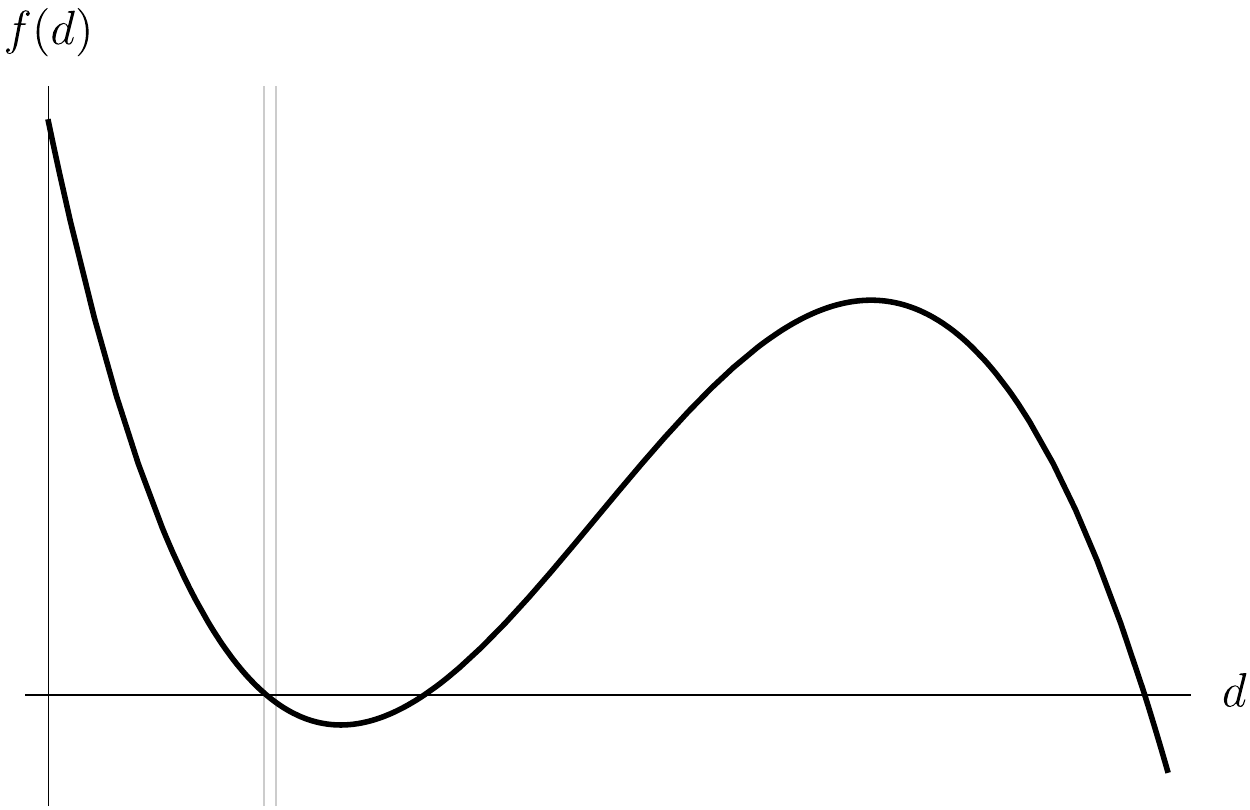}
\endminipage\hfill
\minipage{0.49\textwidth}
\centering
\includegraphics[width=0.75\linewidth]{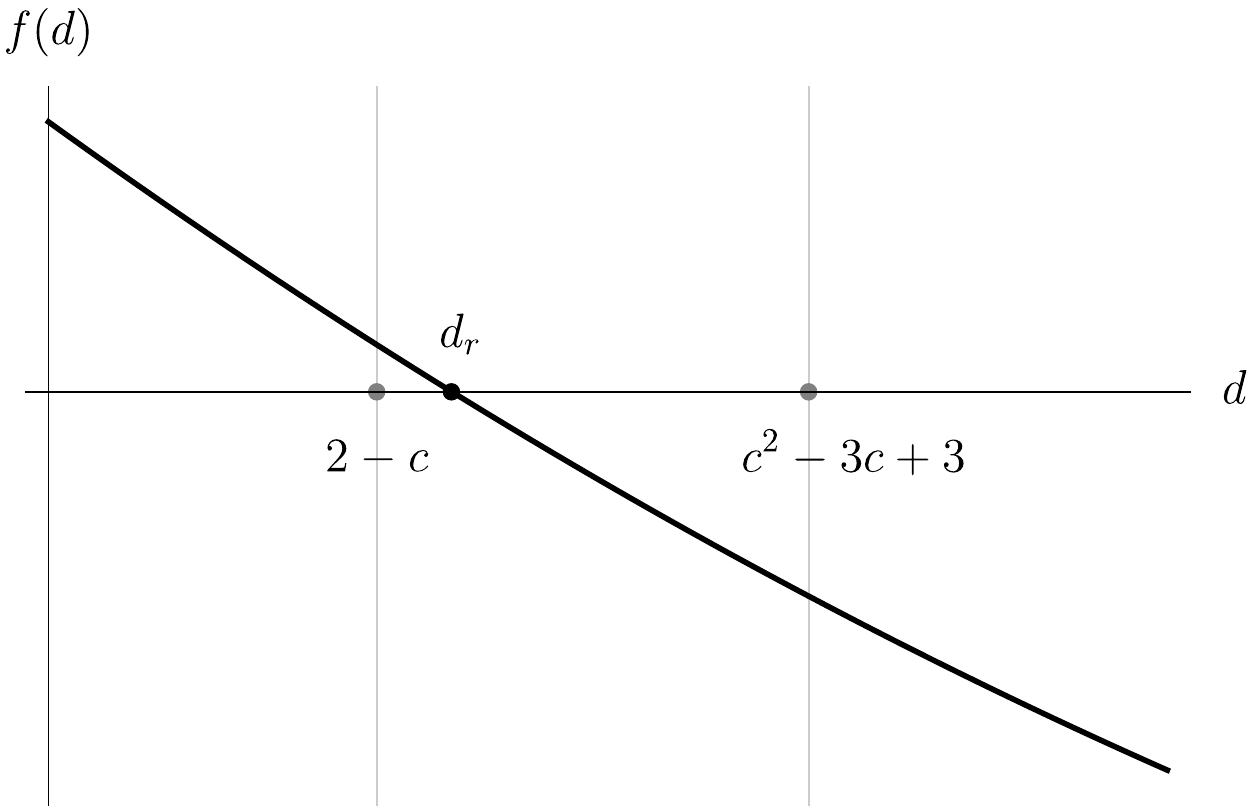}
\endminipage\hfill

\caption{Two characteristic graphs of $f$ (left) and the zoom of their particular segments
between two gray vertical lines involving
the optimal parameter $d_r$ (right). Note that in the neighborhood of $d_r$ the graph of $f$
can be either convex or concave.}\label{fig:quadraticfboth}
\end{figure}

\begin{figure}[h]
  \centering
  \includegraphics[width=0.5\textwidth]{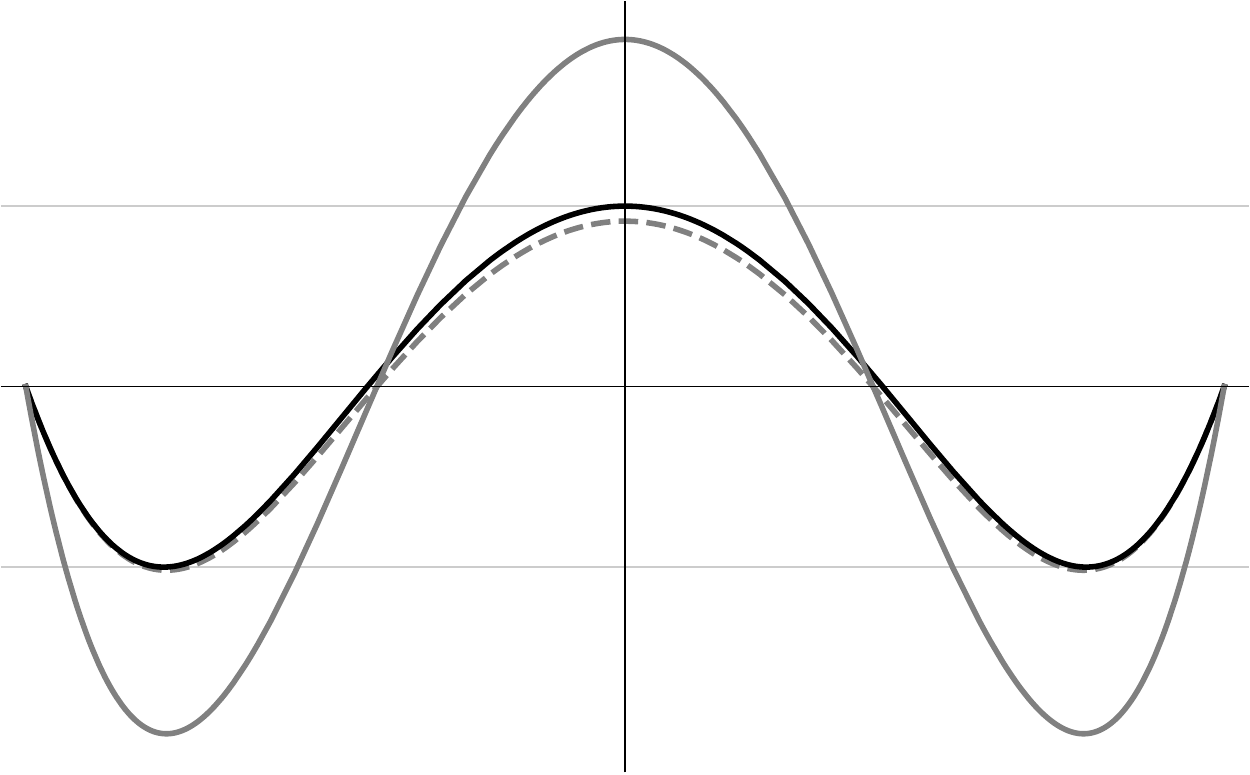}
  \caption{Graphs of the radial error $\phi_{2,0}$ and simplified radial error $\psi_{2,0}$
  in the case of the parabolic $G^0$
  interpolation with $\varphi=\pi/2$.
  The black graph is the optimal radial error $\phi_{2,0}$, the gray one is the optimal simplified radial error
  $\psi_{2,0}$ and the gray dashed is the radial error for the optimal parameter $d_s$ according to the
  simplified radial error. Observe that the minima of the gray dashed graph are smaller than the minima of
  the black one.}
  \label{fig:G20}
\end{figure}

\begin{table}[h]
  \begin{equation*}
    \begin{array}{|c|r|r|r|r|}\hline
        \multicolumn{1}{|c|}{\varphi} & \multicolumn{1}{|c|}{d_r} & \multicolumn{1}{|c|}{\phi_{2,0}}
        & \multicolumn{1}{|c|}{d_s} & \multicolumn{1}{|c|}{\phi_{2,0}}\\ \hline
      \pi/2 & 2.21535 & 1.07676\times 10^{-1} & 2.19737 & 1.09554\times 10^{-1}\\ \hline
      \pi/3 & 1.54728 & 2.36383\times 10^{-2} & 1.54643 & 2.37668\times 10^{-2}\\ \hline
      \pi/4 & 1.30843 & 7.76732\times 10^{-3} & 1.30834 & 7.78280\times 10^{-3}\\ \hline
      \pi/6 & 1.13713 & 1.57677\times 10^{-3} & 1.13712 & 1.57746\times 10^{-3}\\ \hline
      \pi/8 & 1.07713 & 5.03728\times 10^{-4} & 1.07713 & 5.03800\times 10^{-4}\\ \hline
      \pi/12 & 1.03427 & 1.00191\times 10^{-4} & 1.03427 & 1.00194\times 10^{-4}\\ \hline
    \end{array}
  \end{equation*}
  \caption{The table of optimal parameters $d_r$ according to the radial error $\phi_{2,0}$
  and optimal parameters $d_s$ according to the simplified radial error $\psi_{2,0}$ for
  several different inner angles $2\varphi$ of a circular arc $\bfm{c}$. In the last column the values of
  the radial error $\phi_{2,0}$ with respect to $d_s$ are shown. It is clearly seen that
  optimal parameters tend to each other when $\varphi\to0$ but they are not the same.}
  \label{tab:G20}
\end{table}

The parabolic $G^0$ interpolation is just a special case with $k=n-2$.
Note that for $n\geq 2$ the class ${\mathcal G}_n^{n-2}$ is a one-parametric family
of $G^{n-2}$ polynomial interpolants of degree $n$ which will now be studied in detail.

\section{Analysis of $\psi_{n,n-2}$}\label{sec:general_error}

We have seen in the previous section that the analysis of the optimal parabolic $G^0$ interpolant
related to the radial error leads to the analysis of the quartic polynomial \eqref{eq:f(d)}.
It is expected that the problem becomes more and more complicated for higher degrees $n$. Thus
an alternative way of the analysis should be followed.\\
The symmetry and the restriction to $k=n-2$ imply that
a one-parametric family of interpolants has to be studied.
The existence and possible uniqueness of the optimal interpolant in this case has not been proven yet.
It is a natural cornerstone for the study of the multi-parameter family of interpolants
which will be our future work.
We will see that it is  basically enough to study the simplified radial error $\psi_{n,n-2}$,
since its properties can be used to prove similar properties of the radial error $\phi_{n,n-2}$.
The simplified radial error $\psi_{n,n-2}$ becomes
\begin{equation}\label{def:psi1par}
  \psi_{n,n-2}(t)=:\psi_n(t,d)=(t^2-1)^{n-1}q_n(t,d),
\end{equation}
where $d$ is the parameter arising from the $G^{n-2}$ interpolation and $q_n(\cdot,d)$
is an even quadratic polynomial (see, e.g., \cite{Vavpetic-Zagar-general-circle-19}).
Recall again that $q_n$ depends also on $\varphi$, which will not be
explicitly written. Let us first observe the following general result.

\begin{lemma}\label{lem:theta}
Let $\theta_i(t)=(t^2-1)^m r_i(t)$,
where $r_i$ is an even quadratic polynomial, $i=1,2$. If $r_1(0)<r_2(0)$ and $r_1(1)\leq r_2(1)$, then
\begin{enumerate}
  \item $\theta_1(t)>\theta_2(t)$ for all $t\in(-1,1)$, provided $m$ is an odd number,
  \item $\theta_1(t)<\theta_2(t)$ for all $t\in(-1,1)$, provided $m$ is an even number.
\end{enumerate}
\end{lemma}
\begin{proof}
  Let us prove only the first part of the lemma since the proof for the other one is similar.
  For an odd $m$ we observe that $(t^2-1)^m<0$ for all $t\in(-1,1)$.
  By assumption we have $r_1(0)<r_2(0)$ and $r_1(1)\leq r_2(1)$. If there is $\tau\in(-1,1)$ such that
  $r_1(\tau)\geq r_2(\tau)$, the even quadratic polynomials $r_1$ and $r_2$ must have at least four
  intersections on $[-1,1]$ which is a contradiction. Thus $r_1(t)<r_2(t)$ for all $t\in(-1,1)$ and
  consequently $\theta_1(t)>\theta_2(t)$ for all $t\in(-1,1)$.
\end{proof}

The last lemma is crucial in the proof of the existence and the uniqueness
of the best $G^{n-2}$ interpolants with respect to the radial error.
First note that it implies that two graphs of the simplified radial error
\cref{def:psi1par} can not intersect on $(-1,1)$ for two different values $d_1$ and $d_2$,
provided that $q_n(0,d_1)\neq q_n(0,d_2)$ and
\begin{equation}\label{eq:qnineq}
  (q_n(0,d_1)-q_n(0,d_2))(q_n(1,d_1)-q_n(1,d_2))\geq 0.
\end{equation}
The last relation also implies that the same is true for the graphs of the radial error
$\phi_{n,n-2}(t,d)=:\phi_n(t,d)=\sqrt{\psi_n(t,d)+1}-1$ (see \cref{fig:G20}).
By \cite{Eisele-Chebyshev-94} the error function of the best interpolant must alternate
three times, i.e., it must have the unique zero $\tau^*\in(0,1)$ which implies
the smallest possible alternations. The construction of such an interpolant relies on the following
facts:
\begin{itemize}
  \item[(P1)] There exists an interval $I_n\subseteq \RR$ such that for any $\tau\in(0,1)$
  the equation $q_n(\tau,d)=0$ has the unique solution $d\in I_n$. If there is another solution
  $d'\not\in I_n$ of the same equation, then $d$ must induce better interpolant than $d'$.
  \item[(P2)] The functions $q_n(0,\cdot)$ and $q_n(1,\cdot)$ are both increasing or both decreasing functions
  on $I_n$ and the first one must be injective.
\end{itemize}
This suggests the following algorithm for finding $\tau^*$ and consequently
the best interpolant:\\
Choose an arbitrary small $\epsilon>0$ and set $\tau_l:=0$, $\tau_r:=1$.
Let $\tau=\tfrac1 2(\tau_l+\tau_r)$ and let $d_\tau\in I_n$ be the unique solution of
the equation $q_n(\tau,d)=0$ (property (P1)).
If $\max_{t\in[0,1]}\phi_n(t,d_\tau)>|\min_{t\in[0,1]}\phi_n(t,d_\tau)|$,
then set $\tau_l=\tau$, else set $\tau_r=\tau$ (see \cref{fig:bisec}). Repeat this procedure until
$\left|\max_{t\in[0,1]}\phi_n(t,d_\tau)+\min_{t\in[0,1]}\phi_n(t,d_\tau)\right|\leq \epsilon$.\\
This is basically a method of bisection, but since we are interested in finding an interpolant which
alternates, the procedure can be seen as a kind of very well known Remes algorithm
(\cite{Remes-1934_3},\cite{Remes-1934_1},\cite{Remes-1934_2}). Observe that the same algorithm works also
with the simplified error $\psi_n$.

\begin{figure}[h]
  \centering
  \includegraphics[width=0.5\textwidth]{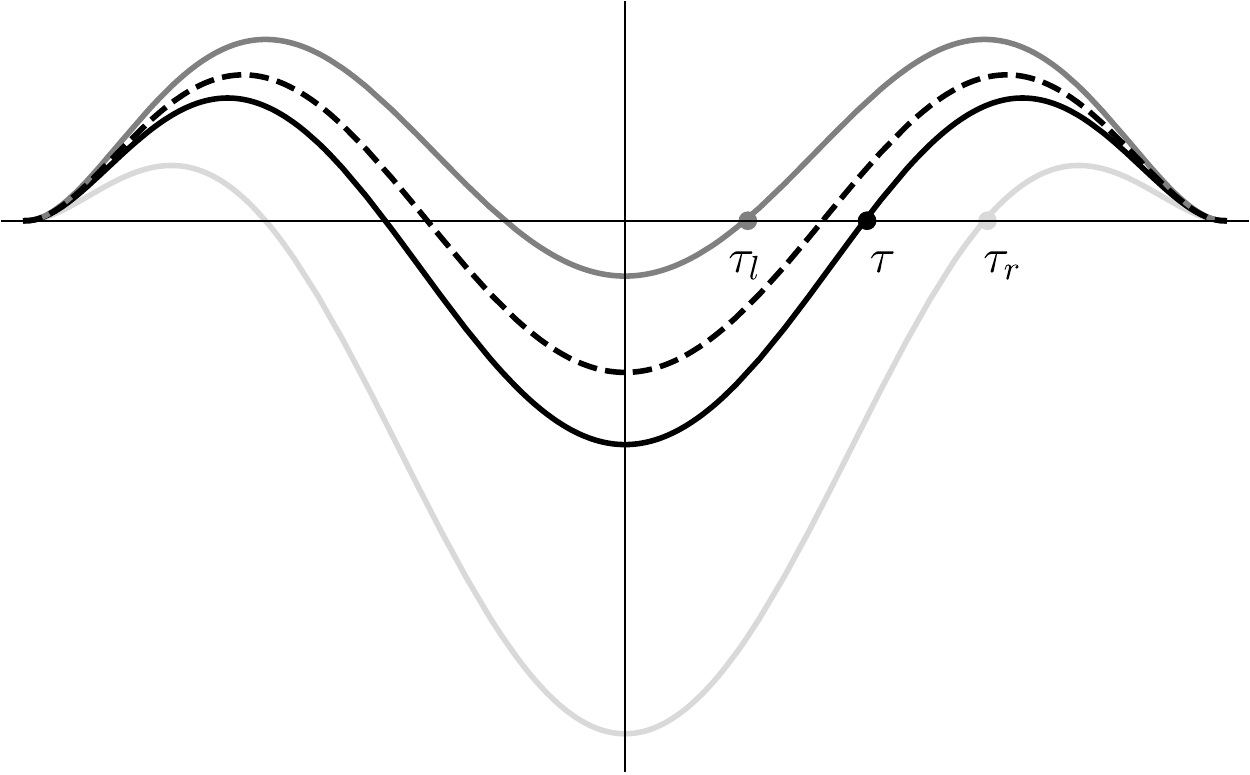}
  \caption{Graphs of $\phi_n(\cdot,d_\tau)$, $\phi_n(\cdot,d_{\tau_l})$ and
  $\phi_n(\cdot,d_{\tau_r})$. The graph of $\phi_n(\cdot,d_{\tau^*})$ is dashed.}
  \label{fig:bisec}
\end{figure}

In the following sections we shall see that the above properties of the polynomial $q_n$
can be proven at least for $n=2,3,4,5$. Consequently the existence and the construction
of the best geometric $G^{n-2}$ interpolant according to the radial error follows.

\section{Parabolic $G^0$ case revisited}\label{sec:parabolic_case_revisited}

The parabolic $G^0$ case was already considered in \cref{sec:parabolicG0direct}.
Here we use the approach described in the previous section.
Recall that the control points of the $G^0$ interpolant
are $\bfm{b}_0=(c,-s)^T$, $\bfm{b}_1=(d,0)^T$ and $\bfm{b}_2=(c,s)^T$.
The corresponding simplified error function is $\psi_2(t,d)=(t^2-1)q_2(t,d)$ and
it is easy to verify that
$$
q_2(t,d)=\frac{1}4 \left((d-c)^2t^2-(d+c)^2+4\right).
$$

\begin{lemma}
  For every $\tau\in[0,1)$ there exists the unique parameter $d_\tau\in I_2:=(1,\infty)$
  such that $q_2(\tau,d_\tau)=0$.
\end{lemma}

\begin{proof}
The lemma follows since $q_2(\tau,1)=\tfrac{1}{4}(1-c)(3+c+(1-c)\tau^2)>0$ and the leading coefficient of the polynomial $q_2(\tau,\cdot)$ is
$-\tfrac 1 4(1-\tau^2)<0$.
\end{proof}

\begin{lemma} The function $q_2(0,\cdot)$ is strictly decreasing and $q_2(1,\cdot)$ is decreasing on $I_2$.
\end{lemma}

\begin{proof}
  The result follows directly from $q_2(0,d)=\tfrac{1}4 \left(4-(d+c)^2\right)$ and $q_2(1,d)=1-c d$.
\end{proof}
Note that the above lemmas imply the properties (P1) and (P2) and the existence and the uniqueness of the best
quadratic $G^2$ interpolant of the circular arc according to the radial error is confirmed.
\section{Cubic $G^1$ case}
The control points in this case are
\begin{equation*}
  \bfm{b}_0=(c,-s)^T,\quad \bfm{b}_1=(c,-s)^T+d\,(s,c)^T,
  \quad \bfm{b}_2=(c,s)^T+d(s,-c)^T,\quad\bfm{b}_3=(c,s)^T,\quad d>0,
\end{equation*} and the corresponding simplified error function is
$\psi_3(t,d)=(t^2-1)^2 q_3(t,d)$, where
$$
q_3(t,d)=\frac{1}{16} \left((3 d c-2s)^2t^2+(3 d s+ 4c)^2-16\right).
$$

\begin{lemma}
For every $\tau\in(0,1)$ there exists the unique parameter $d_\tau\in I_3:=(0,\infty)$ such that $q(\tau,d_\tau)=0$.
\end{lemma}

\begin{proof}
Since $q_3(\tau,\cdot)$ has a positive leading coefficient $\tfrac{9}{16}(s^2+c^2 \tau^2)$
and $q_3(\tau,0)=-\tfrac{1}{4}s^2(4-\tau^2)<0$,
there exists the unique $d_\tau>0$ such that $q_3(\tau,d_\tau)=0$.
\end{proof}

\begin{lemma}
  The functions $q_3(0,\cdot)$ and $q_3(1,\cdot)$ are both strictly increasing functions of $d\in I_3$.
\end{lemma}

\begin{proof}
Lemma follows from $q_3(0,d)=\tfrac{1}{16}\left((3 d s+ 4c)^2-16\right)$ and
$q_3(1,d)=\tfrac{3}{16}\left(3d^2+4s c d-4s^2 \right)$.
\end{proof}

The proof of the existence and the uniqueness of the optimal cubic $G^1$ interpolant of the circular arc is
thus confirmed. An example for the inner angle $\tfrac{\pi}{2}$
together with the corresponding radial error is in \cref{fig:cubicicpi4best}.

\begin{figure}[!htb]
  \minipage{0.49\textwidth}
    \centering
    \includegraphics[width=0.8\linewidth]{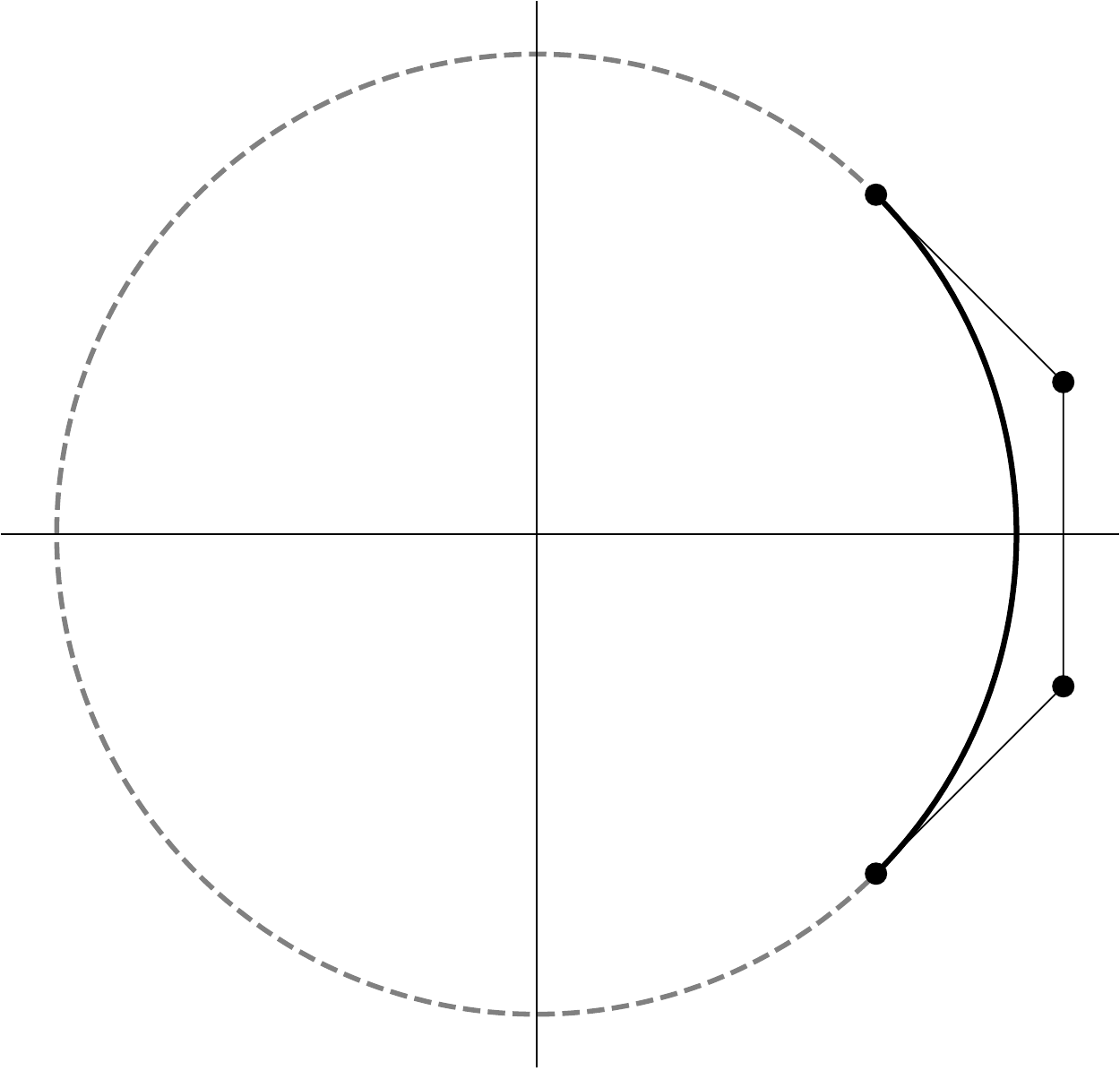}
  \endminipage\hfill
  \minipage{0.49\textwidth}
    \centering
    \includegraphics[width=0.8\linewidth]{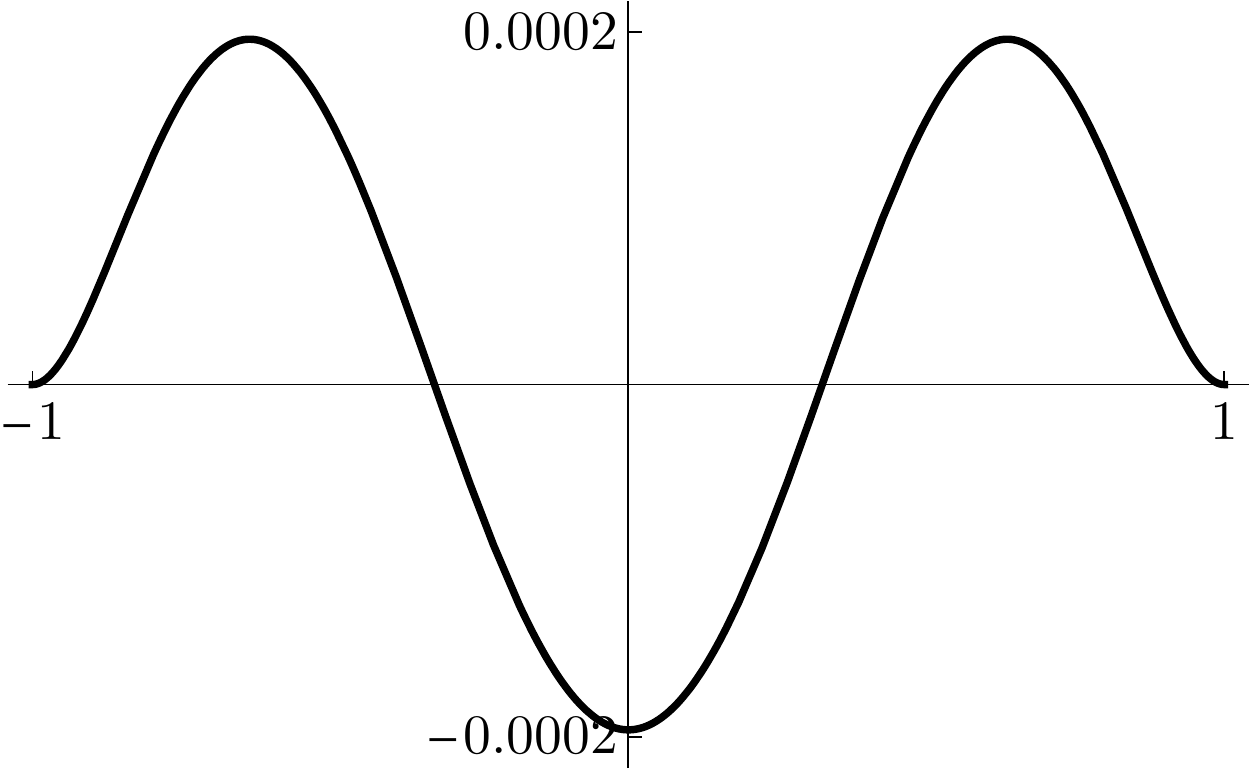}
  \endminipage\hfill
  \caption{Graphs of the best cubic $G^1$ interpolant (left, solid black)
  of the circular arc given by the inner angle $\tfrac{\pi}{2}$ (left, dashed gray) and the corresponding
  radial error (right).}
  \label{fig:cubicicpi4best}
\end{figure}

\section{Quartic $G^2$ case}
Since we are dealing with $G^2$ interpolation, the corresponding control points of the interpolant
$\bfm{p}_4$ read as
\begin{equation*}
  \bfm{b}_0=(c,-s)^T,\ \bfm{b}_1=(c,-s)^T+\frac{\sqrt{3}}{2}\sqrt{1-c\,d}\,(s,c)^T,\
  \bfm{b}_2=\left(d,0\right)^T,\
  \bfm{b}_3=(c,s)^T+\frac{\sqrt{3}}{2}\sqrt{1-c\,d}\,(s,-c)^T,\ \bfm{b}_4=(c,s)^T,
\end{equation*}
$d\in\RR$, and the simplified error function is $\psi_4(t,d)=(t^2-1)^3 q_4(t,d)$, where
\begin{align}
  q_4(t,d)&=\frac{1}{64}
  \left(
  \left(12-3 c^2-30 c d+12 c^3 d+9 d^2+12 \sqrt{3} c s \sqrt{1-c d}
  -12 \sqrt{3} s d \sqrt{1-c d}\right)t^2\right.
  \nonumber\\
  &\left.+52-13 c^2-12 c^3 d-9 d^2-12 \sqrt{3} s d \sqrt{1-c d}
  -2 c \left(9 d+10 \sqrt{3} s \sqrt{1-c d}\right)
  \right).\label{eq:q4}
\end{align}
Note that due to the square root involved we must consider the restriction of the unknown parameter
$d<1/c$. Let us define $a_4:=\frac{\sqrt{3}}{3} s \left(c^2+\sqrt{3} (1-c)\right)$ and
$b_4:=\frac{\sqrt{3}}{3} s\left(\sqrt{3+c^2}-c\right)$. Observe that $0<a_4<b_4<1$ and define
$J_4:=\left(a_4,b_4\right)$.
If $c\neq 0$ then let $I_4:=\left(\frac{1-b_4^2}{c},\frac{1-a_4^2}{c}\right)$
else $I_4:=\left(\tfrac{2}{\sqrt{3}},2\right)$ which can be obtained by taking the limit $c\to 0$.
Observe that $q_4(t,\cdot)$ is an irrational function for $0<c<1$ but with the new variable $x=\sqrt{1-cd}$
the function $r_4(t,x):=q_4(t,\frac{1-x^2}{c})$ becomes a polynomial.
Some of its obvious properties are given by the following lemma.

\begin{lemma}\label{lem:properties_r4}
  For every $t\in(0,1)$ the function $r_4(t,\cdot)$
  \begin{itemize}
    \item[i)] is a quartic polynomial.
    \item[ii)] is increasing on $J_4$ if and only if $q_4(t,\cdot)$ is decreasing on $I_4$.
    \item[iii)] has always four zeros, precisely one on each of the intervals
    $(-\infty,-1)$, $(-1,a_4)$, $J_4$ and $(1,\infty)$.
  \end{itemize}
\end{lemma}

\begin{proof}
  The first two properties are obvious. To see the third one (see \cref{fig:r4_general}),
  first observe that the leading coefficient
  of $r_4(t,\cdot)$ is $lc(t)=-\frac{9}{64c^2}(1-t^2)$ and it is obviously negative.
  \begin{figure}[h]
  \centering
  \includegraphics[width=0.5\textwidth]{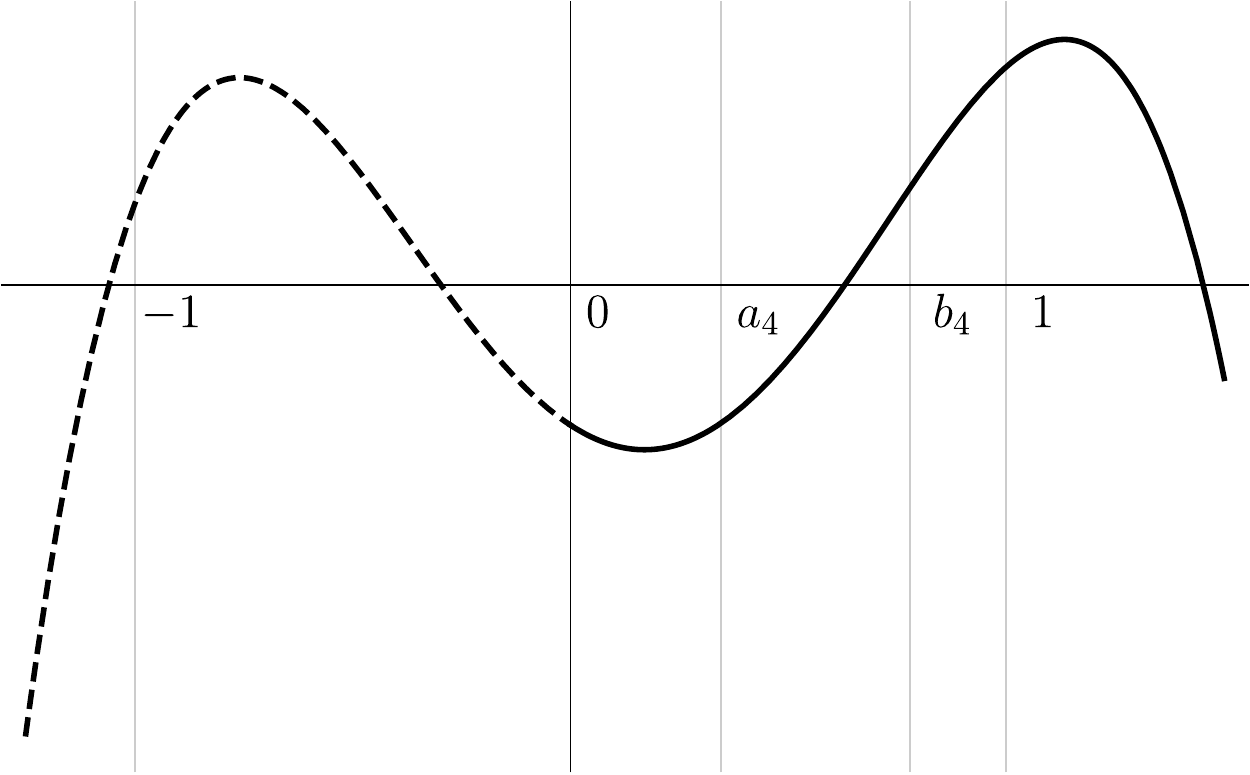}
  \caption{A characteristic graph of $r_4(t,\cdot)$. }
  \label{fig:r4_general}
\end{figure}
  A straightforward calculation reveals that
  \begin{align*}
    r_4(t,-1)&=\frac{1}{64} \left(39+8 \sqrt{3} c s+9 t^2+12 \sqrt{3} c s
    \left(1-t^2\right)+\left(1-c^2\right) \left(13+3t^2\right)\right)>0,\\
    r_4(t,b_4)&=\frac{1}{4} \left(1-c^2\right) \left(c \sqrt{3+c^2}-1-c^2\right)^2>0,\\
    r_4(t,1)&=\frac{1}{64} \left(52+12 t^2
   +4 \sqrt{3} c s \left(-5+3 t^2\right)
   -c^2 \left(13+3 t^2\right)\right)\geq
   \frac{1}{64} \left(52-20\sqrt{3}-16\right)>0.
  \end{align*}
Since $lc(t)<0$ the property (iii) will follow if we verify $r_4(t,a_4)<0$.
In order to do this let us first observe that
\begin{equation}
  r_4(t,a_4)=-\frac{1}{64}(1-c)^3(1+c)s_4(t),\label{r4s4}
\end{equation}
where $s_4$ is an even quadratic polynomial. It is also quite easy to verify
\begin{align*}
  s_4(0)&=\left(4-2 \left(1+\sqrt{3}\right) c^2\right)^2+12 \left(2+\sqrt{3}\right) c^3 (1-c)
   +20\left(\sqrt{3}-1\right) c
   \left(1-c^3\right)\\
   &+\left(4 \left(7-4 \sqrt{3}\right)+c^2\right) \left(1-c^4\right)+20 \left(2 \sqrt{3}-3\right)+2
   \left(5-2 \sqrt{3}\right) c^3+2 \left(1+2 \sqrt{3}\right) c^5>0,\\
  s_4(1)&=2 \left(\left(7-4 \sqrt{3}\right) \left(64+40 \sqrt{3}-7 c\right)+c \left(3-2 \sqrt{3} c\right)^2+c^3+c^3 \left(3\sqrt{3}-2-2 c\right)^2\right)>0,
\end{align*}
and the proof of the lemma is complete.
\end{proof}

\begin{lemma}\label{lem:P1_quartic}
For every $\tau\in(0,1)$ there exists the unique parameter
$d_\tau\in I_4$ such that $q_4(\tau,d_\tau)=0$.
\end{lemma}
\begin{proof}
  If $0<c<1$, the result follows directly from the previous lemma.
  If $c=0$ then $I_4=\left(\frac{2}{\sqrt{3}},2\right)$, $q_4(\tau,\cdot)$ is a quadratic polynomial
with $q_4(\tau,\frac{2}{\sqrt{3}})q_4(\tau,2)=
-\frac{1}{32} \left(6 \sqrt{3}-8+3 \left(2-\sqrt{3}\right) \left(1-\tau^2\right)\right)<0$
and the result of the lemma follows.
\end{proof}
Since $q_4(t,d)=r_4(t,\sqrt{1-cd})$, each nonnegative zero of $r_4(t,\cdot)$ implies
one zero of $q_4(t,\cdot)$. Thus by \cref{lem:properties_r4} $q_4(t,\cdot)$ might have
two or three zeros. But the next lemma reveals that the one on $I_4$ provides the best approximant.
\begin{lemma}\label{lem:lc_quartic}
  Let $d_\tau^*$ be the unique solution of $q_4(\tau,d)=0$ on $I_4$. If $d_\tau$ is
  any other solution of the same equation then the absolute value of the
  leading coefficient of $q_4(\cdot,d_\tau)$ is bigger than the absolute value of the leading
  coefficient of $q_4(\cdot,d_\tau^*)$.
\end{lemma}
\begin{proof}
  If $c=0$ then $I_4=\left(\frac{2}{\sqrt{3}},2\right)$ and $q_4(\tau,\cdot)$ is a quadratic polynomial
  with the leading coefficient $-\tfrac{9}{64}(1-\tau^2)<0$. Furthermore,
  $q_4(\tau,0)=\tfrac{1}{16}\left(13-3\tau^2\right)>0$ thus the additional solution $d_\tau$ of
  $q_4(\tau,d)=0$ is negative. Since the leading coefficient
  $lc(d)=\frac{3}{64} \left(2-\sqrt{3}d\right)^2$ of $q_4(\cdot,d)$ is
  nonnegative, increasing on $(\tfrac{2}{\sqrt{3}},\infty)$ and decreasing elsewhere, we have
  $lc(d_\tau)>lc(0)=\tfrac{3}{16}>\tfrac{3}{8}(2-\sqrt{3})=lc(2)>lc(d_\tau^*)$.\\
  Let us assume now that $0<c<1$. Denote by $x_\tau^*=\sqrt{1-cd_\tau^*}$ the unique solution
  of $r_4(\tau,x)=0$ on $J_4$ and by $x_\tau=\sqrt{1-cd_\tau}$ any other solution of the same equation.
  The leading coefficient $lc$ of the polynomial $r_4(\cdot,x)$ is
  $$
    lc(x)=\frac{1}{64c^2}
    \left(9 x^4+12 \sqrt{3} c s x^3-6 \left(3-5 c^2+2 c^4\right) x^2-12 \sqrt{3} c s^3 x+9 s^4\right).
  $$
 It can be verified that $lc\geq 0$
 with two double zeros $b_4^{-}=-\frac{\sqrt{3}}{3}s(\sqrt{3+c^2}+c)$ and $b_4$, and
 a unique local maximum at $-\tfrac{cs}{\sqrt{3}}$
 (see \cref{fig:lc_general}). Thus $lc$ is decreasing on $[0,b_4)$ and increasing
 on $(b_4,\infty)$.

 \begin{figure}[h]
  \centering
  \includegraphics[width=0.5\textwidth]{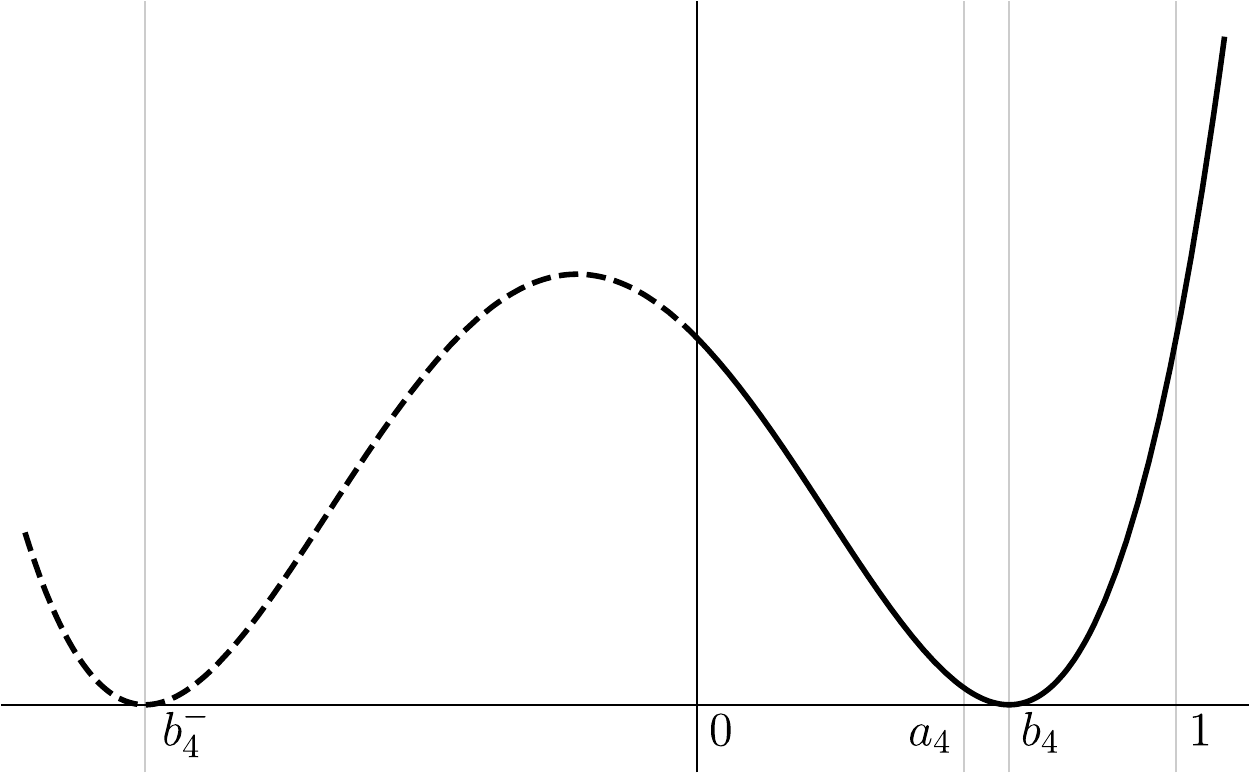}
  \caption{A characteristic graph of the leading coefficient of $r_4(\cdot,d)$.
  It is non-negative with the double zeros at
  $b_4^{-}$ and at $b_4$.}
  \label{fig:lc_general}
\end{figure}

 \noindent The proof of the lemma will be complete if we show $lc(x_\tau^*)<lc(x_\tau)$.
 If $x_\tau<a_4$ then the result follows from the fact that $lc$ decreases on $[0,b_4)$.
 If $x_\tau>b_4$, then $x_\tau>1$ by \cref{lem:properties_r4} and it is enough to
 verify that $lc(1)\geq lc(a_4)$. Because
 \begin{align*}
   lc(1)-lc(a_4)&=
   c \left(4 \sqrt{3}-9+4 c^2\right)^2+12 \sqrt{3} c \sqrt{1-c^2}+\left(24 \sqrt{3}-36\right) (1-c)
   +3 \left(28\sqrt{3}-43\right) c (1-c)\\
   &+\left(3+4 \sqrt{3}\right) c^2 (1-c)+\left(3+4 \sqrt{3}\right) c^3 (1-c)+5 \left(20
   \sqrt{3}-29\right) c^4 (1-c)+\left(76 \sqrt{3}-129\right) c^5 (1-c)\\
   &+\left(28 \sqrt{3}-47\right) c^6 (1-c)+\left(28
   \sqrt{3}-39\right) c^7 (1-c)+\left(36 \sqrt{3}-59\right) c^8 (1-c)+\left(40 \sqrt{3}-63\right) c^9 (1-c)\\
   &+8 \left(5
   \sqrt{3}-8\right) c^{10}>0,
 \end{align*}
 the proof of the lemma is complete.
 \end{proof}
From the last two lemmas the property (P1) follows. In order to confirm also the property (P2),
we additionally have to prove the following lemma.

\begin{lemma}\label{lem:P2_quartic}
  The functions $q_4(0,\cdot)$ and $q_4(1,\cdot)$ are strictly decreasing on $I_4$.
\end{lemma}
\begin{proof}
  Let us first observe that for $c=0$ we have $q_4(0,d)=\frac{1}{64} \left(52-12 \sqrt{3} d-9 d^2\right)$
  and $q_4(1,d)=1-\frac{3 \sqrt{3} d}{8}$ which are obviously strictly decreasing on $I_4$.\\
  If $0<c<1$ we again consider the function $r_4(t,x)=q_4\left(t,\frac{1-x^2}{c}\right)$.
  By \cref{lem:properties_r4} the proof will be complete if we verify that
  $r_4(0,\cdot)$ and $r_4(1,\cdot)$ are strictly increasing on $J_4$.\\
  Recall that $r_4(t,\cdot)$ is a quartic polynomial with the leading coefficient
  $-\frac{9}{64 c^2}\left(1-t^2\right)$. Consequently, $f_0(x):=\frac{dr_4}{dx}(0,x)$ is
  a cubic polynomial with the negative leading coefficient and
  $f_1(x):=\frac{dr_4}{dx}(1,x)$ is quadratic with the positive leading coefficient
  $\tfrac{9\sqrt{3}s}{8c}$. So, it suffices to show that
  $f_i(0)<0$, $f_i(a_4)>0$ and $f_i(1)>0$ for $i=0,1$. By straightforward computations we have
  \begin{equation*}
    f_0(0)=-\frac{\sqrt{3} \left(3+5 c^2\right) s}{16 c}<0,\quad
    f_1(0)=-\frac{\sqrt{3} \left(3+c^2\right) s}{8 c}<0,\quad
  \end{equation*}
  and
  \begin{equation*}
    f_0(1)=\frac{9c+6c^3+\sqrt{3}s(6-5c^2)}{16c}>0,\quad
    f_1(1)=\frac{12c+\sqrt{3}s\left(6-c^2\right)}{8c}>0.
  \end{equation*}
  In order to show $f_i(a_4)>0$, $i=0,1$, we observe that $f_0(a_4)=\frac{s(1-c)}{16c}\,g_0(c)$
  and $f_1(a_4)=\frac{\sqrt{3}s(1-c)}{8c}g_1(c)$,
  where
  \begin{align*}
    g_0(c)&=6 \left(3+\sqrt{3}\right)+9 \left(1-2 \sqrt{3}\right) c-5 \sqrt{3} c^2+2
   \left(3+8 \sqrt{3}\right) c^3-2 \left(9+4 \sqrt{3}\right)
   c^4+\left(9+2 \sqrt{3}\right) c^5-\sqrt{3} c^6,\\
   g_1(c)&=6-4 \left(3-\sqrt{3}\right) c-\left(13-6 \sqrt{3}\right) c^2+9
   c^3+\left(3-6 \sqrt{3}\right) c^4+3 c^5.
  \end{align*}
  Let us write $g_0=g_{01}+g_{02}+g_{03}$, where
  \begin{align*}
    g_{01}(c)&=6 \left(3+\sqrt{3}\right)-9 \left(-1+2 \sqrt{3}\right) c,\\
    g_{02}(c)&=-5 \sqrt{3} c^2+\left(9+\sqrt{3}\right) c^5,\\
    g_{03}(c)&=c^3 \left(2 \left(3+8 \sqrt{3}\right)-2 \left(9+4 \sqrt{3}\right)
   c+\sqrt{3} c^2 (1-c)\right).
  \end{align*}
  It is now easy to verify $g_{01}(c)>27-12\sqrt{3}\approx 6.22$,
  $g_{02}(c)\geq -3 \sqrt[6]{3^{5}}\sqrt[3]{\left(\frac{2}{9+\sqrt{3}}\right)^{2}}\approx-2.45$,
  $g_{03}(c)>0$, thus $g_0(c)>0$ and consequently $f_0(a_4)>0$. Similarly, let us write
  $g_1=g_{11}+g_{12}$, where
  \begin{equation*}
    g_{11}(c)=6-4 \left(3-\sqrt{3}\right) c\quad {\rm and}\quad
    g_{12}(c)=c^2 \left(-13+6 \sqrt{3}+9 c+\left(3-6 \sqrt{3}\right) c^2+3 c^3\right).
  \end{equation*}
  Again, we can bound $g_{11}(c)>4\sqrt{3}-6\approx 0.93$ and $g_{12}(c)>-0.05$, thus $g_1(c)>0$ and
  so $f_1(a_4)>0$. Note that the exact value of $g_{12}$ can be given in a closed form but it is
  to complicated to be written here.
\end{proof}

From the previous proofs the properties (P1) and (P2) follow also for the quartic $G^2$ case.
Let us demonstrate the results by some examples. Consider the quartic case with
$\varphi=\pi/4$. According to \cref{lem:P1_quartic}, \cref{lem:lc_quartic} and
\cref{lem:P2_quartic}, there is the unique best $G^2$ interpolant of the circular arc with respect to
the radial error. From the above proofs of the lemmas it can be seen that there might be another
$G^2$ interpolant for which the radial error alternates three times. More precisely, some analysis
based on numerical computations actually reveals that this happens for all angles $\varphi<\arccos{3/5}$.
Indeed, such an interpolant exists
but its radial error is much bigger that the error of the optimal interpolant. Numerical experiments
indicate that it is quite difficult to find the best interpolant numerically directly by constructing
an appropriate equation for the unknown $d$. The bisection algorithm described in \cref{sec:general_error}
should be used instead. The results are shown in \cref{fig:quarticpi4}.

\begin{figure}[!htb]
  \minipage{0.49\textwidth}
    \centering
    \includegraphics[width=0.8 \linewidth]{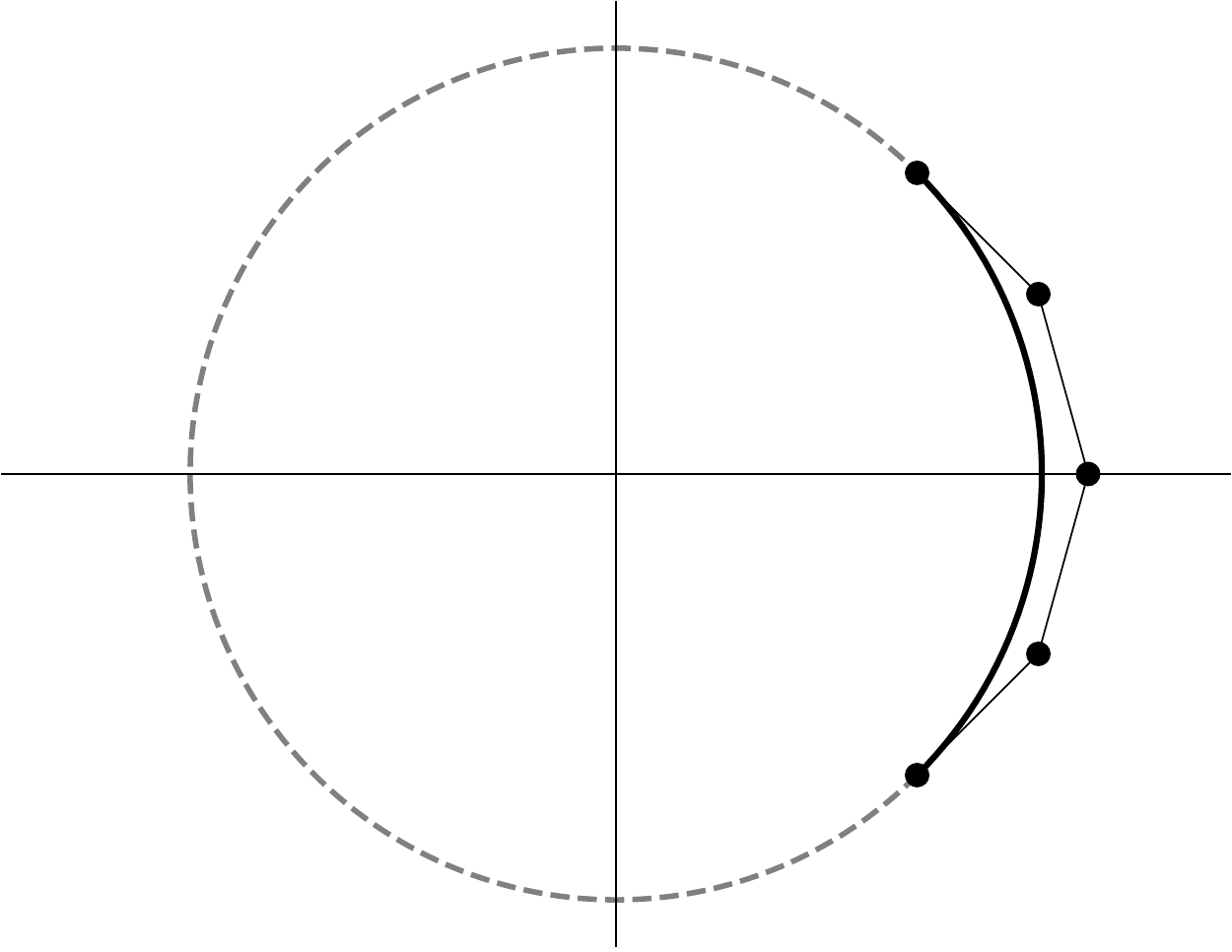}
  \endminipage\hfill
  \minipage{0.49\textwidth}
    \centering
    \includegraphics[width=0.8\linewidth]{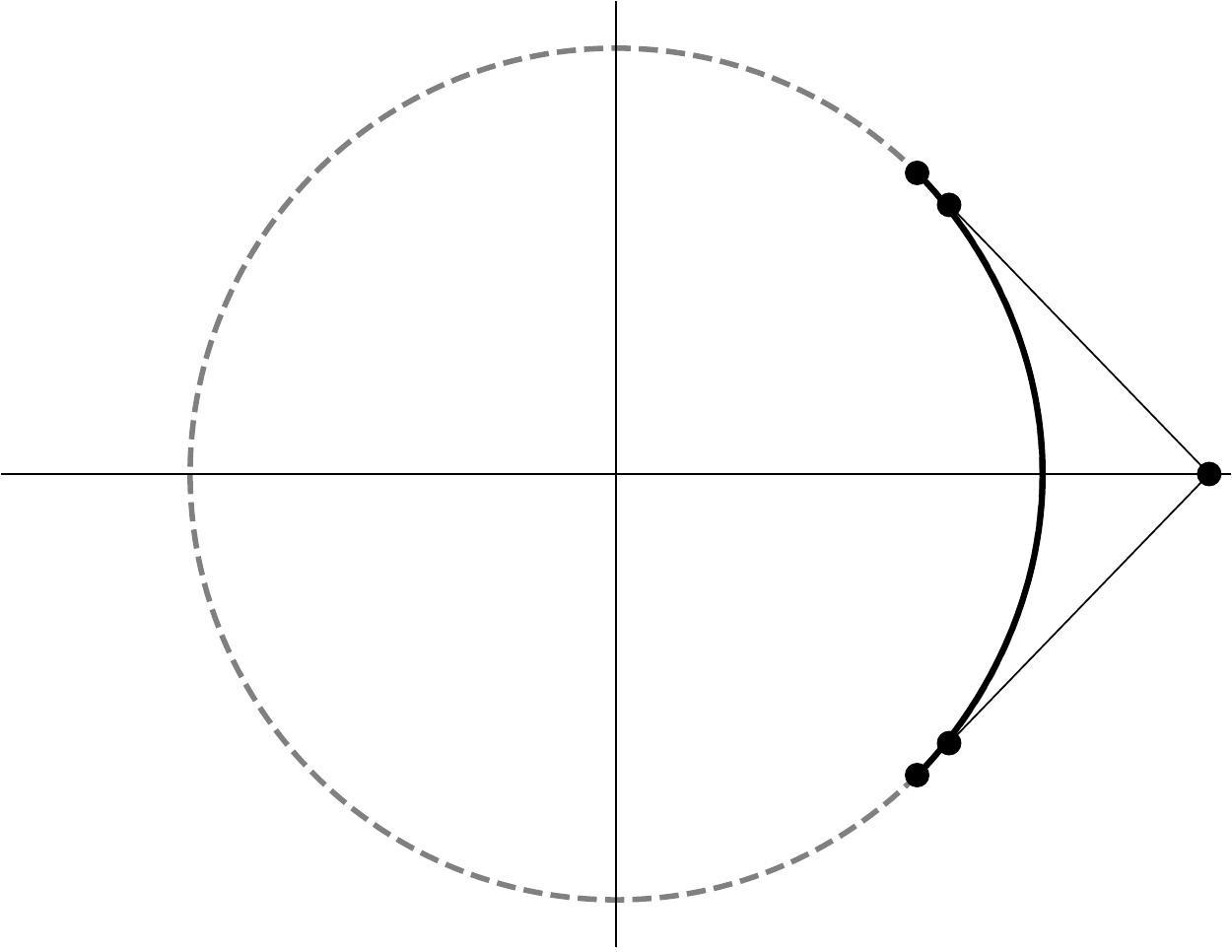}
  \endminipage\hfill
  \caption{Graphs of two quartic $G^2$ interpolants (solid black)
  of the circular arc given by the inner angle $\pi/2$ (dashed gray). The best one is on the left.
  The maximum of its radial error is approximately $2.5\times 10^{-6}$.
  The interpolant on the right has alternating radial error too,
  but its maximal value is approximately  $2\times 10^{-3}$ and it is much bigger than for the optimal interpolant.}
  \label{fig:quarticpi4}
\end{figure}

\section{Quintic $G^3$ case}
Here we consider the $G^3$ interpolation which implies the control points
\begin{align*}
  \bfm{b}_0&=(c,-s)^T,\quad
  \bfm{b}_1=(c,-s)^T+d(s,c)^T,\quad
  \bfm{b}_2=\left(\frac{5d(4-5d^2)c+4(2+5d^2)s}{4(5d+2 s c)},-\frac{5d((4-5d^2)s-6 d c)}{4(5d+2 s c)}\right)^T,\\
  \bfm{b}_3&=\left(\frac{5d(4-5d^2)c+4(2+5d^2)s}{4(5d+2 s c)},\frac{5d((4-5d^2)s-6 d c)}{4(5d+2 s c)}\right)^T,\quad
  \bfm{b}_4=(c,s)^T+d(s,-c)^T\quad \bfm{b}_5=(c,s)^T,
\end{align*}
with $d<0$ due to the $G^1$ continuity. The corresponding simplified error function is
$\psi_5(t,d)=(t^2-1)^4 q_5(t,d)$, where $q_5(\cdot,d)$ is a quadratic polynomial with rational coefficients in $d$.
Let us define
\begin{equation*}
  0<\alpha_5:=\frac{2s}{5}\leq a_5:=\frac{16s}{25+15c}<b_5:=\frac{2s}{3+2c}
  <\beta_5:=\frac{4s}{5(1+c)}<1,
\end{equation*}
and $I_5:=(a_5,b_5)$.
Similarly as in the previous cases, we will confirm the properties (P1) and (P2) (see \cref{fig:quintic_general}
for better understanding the proofs in the following).

\begin{figure}[h]
  \centering
  \includegraphics[width=0.5\textwidth]{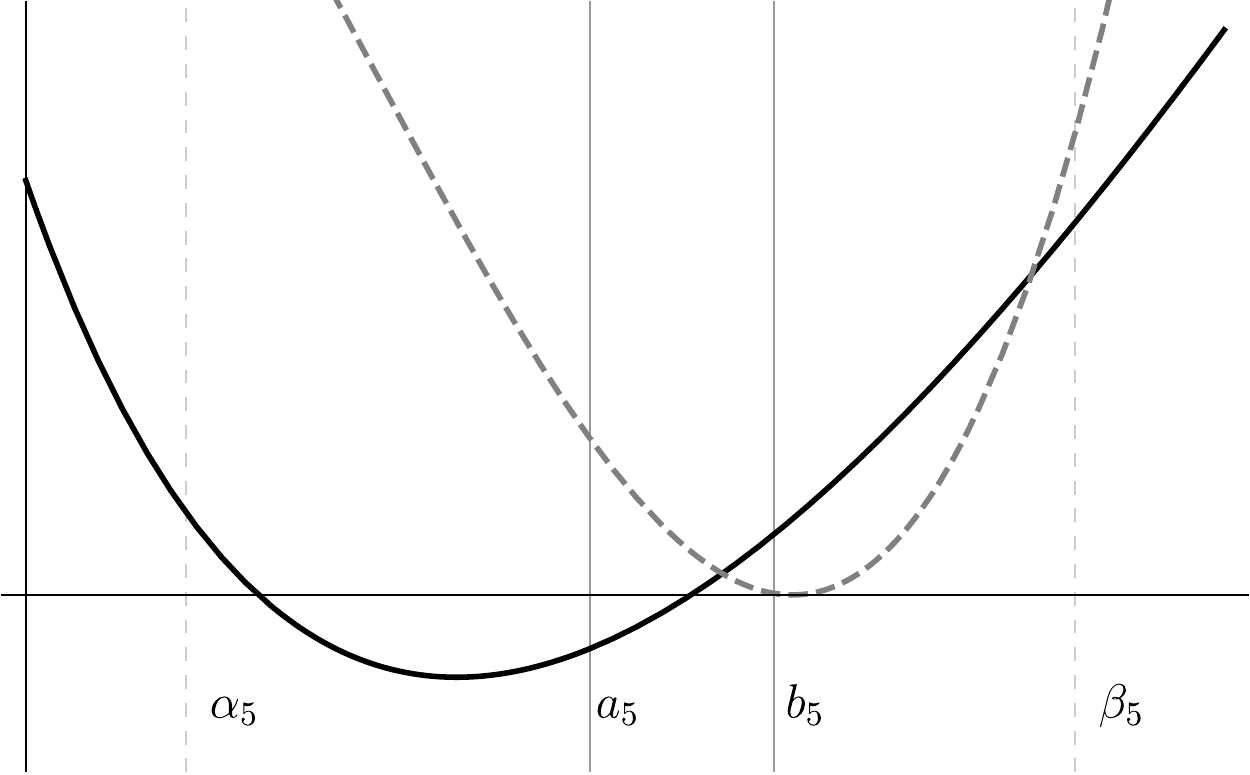}
  \caption{A characteristic graph of $q_5(\cdot,d)$ (black) and its leading coefficient
  $lc$ (gray dashed).}
  \label{fig:quintic_general}
\end{figure}

Let us first consider the behavior of the leading coefficient
$lc$ of $q_5(\cdot,d)$ which is of the form $lc(d)=(r(d))^2$, where
\begin{equation*}\label{eq:f}
 r(d)=\frac{125 s d^3+100c d^2-20s(3+c^2)d+16 c s^2}{32(5d+2 s c)}.
\end{equation*}

\begin{lemma}\label{quintic_leading}
The leading coefficient $lc$ of the polynomial $q_5(\cdot,d)$ is a decreasing function of the variable $d$
on $\left(\alpha_5,b_5\right)\supset I_5$ and it is an increasing function on
$\left(\beta_5,\infty\right)$.
\end{lemma}

\begin{proof}
Recall that $lc(d)=(r(d))^2$.
The denominator of the derivative
$$
r'(d)=\frac{5(125 s d^3+25c(2+3s^2)d^2+40 s c^2 d - 4 c s^2(5+c^2))}{16(5d+2 s c)^2}
$$
is positive and the nominator is an increasing function of the variable $d\in[\alpha_5,\infty)$.
Since $r'(\alpha_5)=\tfrac{5(1-c)}{8(1+c)^2}(1+c+2c^2)>0$, the function $r$ is increasing on
$(\alpha_5,\infty)$. Since
\begin{equation*}
r\left(b_5\right)=-\frac{(1-c)^2(5-2c+2c^2)s}{4(3+2c)^2(5+3c+2c^2)}<0 \quad {\rm and}\quad
r\left(\beta_5\right)=\frac{(1-c)^2s}{4(1+c)(2+c+c^2)}>0.
\end{equation*}
the result of the lemma follows.
\end{proof}

\begin{lemma}
For any $t\in(0,1)$ the function $q_5(t,\cdot)$ is a convex function on the interval
$(0,\beta_5)\subset (0,1)$.
\end{lemma}

\begin{proof}
Let us define $g(x)=q_5(t,\beta_5 x)$. Then the function $q_5(t,\cdot)$ is convex
on the interval $(0,\beta_5)$ if and only if the function $g$ is convex on the interval $(0,1)$.
Since
$$
g''(x)=\frac{1-c}{8(1+c)^2 \left(2 x+c+c^2\right)^4}
\left(\sum_{j=0}^5a_j(t)x^j(1-x)+a_6(t)x^6\right),
$$
for some even quadratic polynomials $a_j$, it is enough to show that $a_j(0)>0$ and $a_j(1)>0$
for $j=0,\ldots,6$. The latter can be easily checked and details will be omitted here.

\end{proof}

\begin{lemma}\label{quintic_uniquness}
  For any $t\in(0,1)$ the function  $\psi_5(t,\cdot)$ has exactly two zeros on the interval
  $(0,\beta_5)$. One is on the interval $(\alpha_5,a_5)$ and the other one
  is on the interval $I_5$.
\end{lemma}

\begin{proof}
By the previous lemma $\psi_5(t,\cdot)$ has at most two zeros on the interval $(0,\beta_5)$.
The result of the lemma follows since
\begin{align*}
\psi_5\left(t,\alpha_5\right)&=\frac{(1-c)^5(1-t^2)^4t^2}{16(1+c)}>0,\\
\psi_5\left(t,a_5\right)&=-\frac{(1+c)(1-c)^4(1-t^2)^4}{1024(5+3c)^4(8+5c+3c^2)^2}\\
&\times(48 (15925 + 4065 c + 1338 c^2 - 3846 c^3 - 1071 c^4 - 27 c^5) -
 64 (1 - c) (44 + 3 c + 9 c^2)^2 t^2)<0,\\
\psi_5\left(t,b_5\right)&=\frac{(1+c)(1-c)^4(1-t^2)^4}{64(3+2c)^4(5+3c+2c^2)^2}
(16 c (180 + 169 c + 123 c^2 + 27 c^3 + c^4) + 4 (1 - c) (5 - 2 c + 2 c^2)^2 t^2)>0.
\end{align*}
\end{proof}

We are finally ready to confirm the properties (P1) and (P2).

\begin{lemma}
  The functions $q_5(0,\cdot)$ and $q_5(1,\cdot)$ are strictly increasing on $I_5$.
  For every $\tau\in(0,1)$ there exists the unique solution $d_\tau$ of the equation
  $q_5(\tau,d)=0$ on the interval $I_5$. For any other solution $d$ of the equation $q_5(t,d)=0$,
  we have $lc(d)>lc(d_t)$.
\end{lemma}

\begin{proof}
Let us first prove the monotonicity of $q_5(0,d)$. Note that $q_5(0,d)=\psi_5(0,d)=(f(d))^2-1$, where
$$
f(d)=\frac{8(5+3c^2)s+20c(9-c^2)d+150s d^2-125c d^3}{32(5d+2sc)}.
$$
Since $I_5\subset (0,1)$ and $f>0$ on $(0,1)$, it is enough to show that
$f'>0$ on $I_5$. The derivative $f'$ can be written as
$f'(d)=\tfrac5{16}(5d+2sc)^{-2}g(d)$, where $g(d)=-4(5-6c^2+c^4)s+60s^2cd+75s^3d^2-125cd^3$.
The inequality $g(d)>0$ holds for all $d\in I_5$ since
\begin{align*}
&g\left(\frac{2(1-c)s}{75+95c+30c^2}x+\frac{16s}{25+15c}\right)=\frac{4s^3(1-c)^2}{(3+2c)^3(5+3c)^3}
\left((3+2c)^3(335+297c+189c^2+27c^3)\right.\\
&\left.+30(3+2c)^2(8+5c+3c^2)x+3(5+2c+c^2)(3+2c)(1-3c)x^2-2(1-c)cx^3\right)\\
&\geq\frac{4s^3(1-c)^2}{(3+2c)^3(5+3c)^3}\left(3^3\,335+0-90-2\right)>0
\end{align*}
for all $x\in (0,1)$.

Let us now check that the function $h=q_5(1,\cdot)$ is increasing on $I_5$.
Note that
\begin{align*}
h'\left(a_5\right)&=\frac{5 (1-c)^2 s}{64 (5+3 c)^3 \left(8+5 c+3 c^2\right)^3}\\
&\times \left(440863+1380116 c+1950998 c^2+1731140 c^3+984672 c^4+368604 c^5+80946 c^6+8748 c^7+729
   c^8\right)>0,\\
h''\left(b_5\right)&=\frac{375 (1-c)}{32 (3+2
   c)^2 \left(5+3 c+2 c^2\right)^4}\\
   &\times \left(2755+9325 c+15649 c^2+16453 c^3+11136 c^4+4638 c^5+688 c^6-356 c^7-240 c^8-48 c^9\right)>0,\\
h'''\left(b_5\right)&=-\frac{375}{16 (3+2 c) \left(5+3 c+2 c^2\right)^5 s}\\
&\times\left(650+8850 c+24645 c^2+41580 c^3+53039 c^4+52746 c^5+40515 c^6+22080 c^7+7255 c^8+216 c^9\right.\\
&\left.-1016 c^{10}-480 c^{11}-80 c^{12}\right)<0,
\end{align*}
and
$$
  h''''\left(d\right)=\frac{1875}{128} \left(1+\frac{320 \left(5+2 c^2+c^4\right)^2 s^4}{(5 d+2 c s)^6}\right)>0.
$$
These imply that $h'''<0$, $h''>0$ and $h'>0$ on $I_5$.

By Lemma~\ref{quintic_uniquness}, there is the unique solution $d_\tau\in I_5$ of the equation
$q_5(\tau,d)=0$ and for any other solution $d$ we have $d<d_\tau$ or $d>\beta_5$.
By Lemma~\ref{quintic_leading}, the leading coefficient $lc(d)>lc(a_5)>lc(d_\tau)$,
provided $d<d_\tau$, and $lc(d)>lc(\beta_5)$, provided $d>\beta_5$.
Hence it is enough to show that $lc(\beta_5)>lc(a_5)$, which follows from
\begin{align*}
\sqrt{lc\left(\beta_5\right)}-\sqrt{lc\left(a_5\right)}=r(\beta_5)-r(a_5)
&=\frac{s(1-c)^2}{4(1+c)(2+c+c^2)}-\frac{s(1-c)^2(44 + 3 c + 9 c^2)}{4 (5 + 3 c)^2 (8 + 5 c + 3 c^2)}\\
&=\frac{s(1-c)^2\left(112+227c+182c^2+58c^3+6c^4-9c^5\right)}{4\left(1+c\right)\left(2+c+c^2\right)\left(5 + 3 c\right)^2 \left(8 + 5 c + 3 c^2\right)}>0.
\end{align*}
\end{proof}

The following example reveals that there might be several quintic $G^3$ interpolants with the alternating
radial error. Some of them possess quite interesting configuration of the control points
(see \cref{fig:quinticpi4}).

\begin{figure}[!htb]
  \minipage{0.33\textwidth}
    \centering
    \includegraphics[width=1.0 \linewidth]{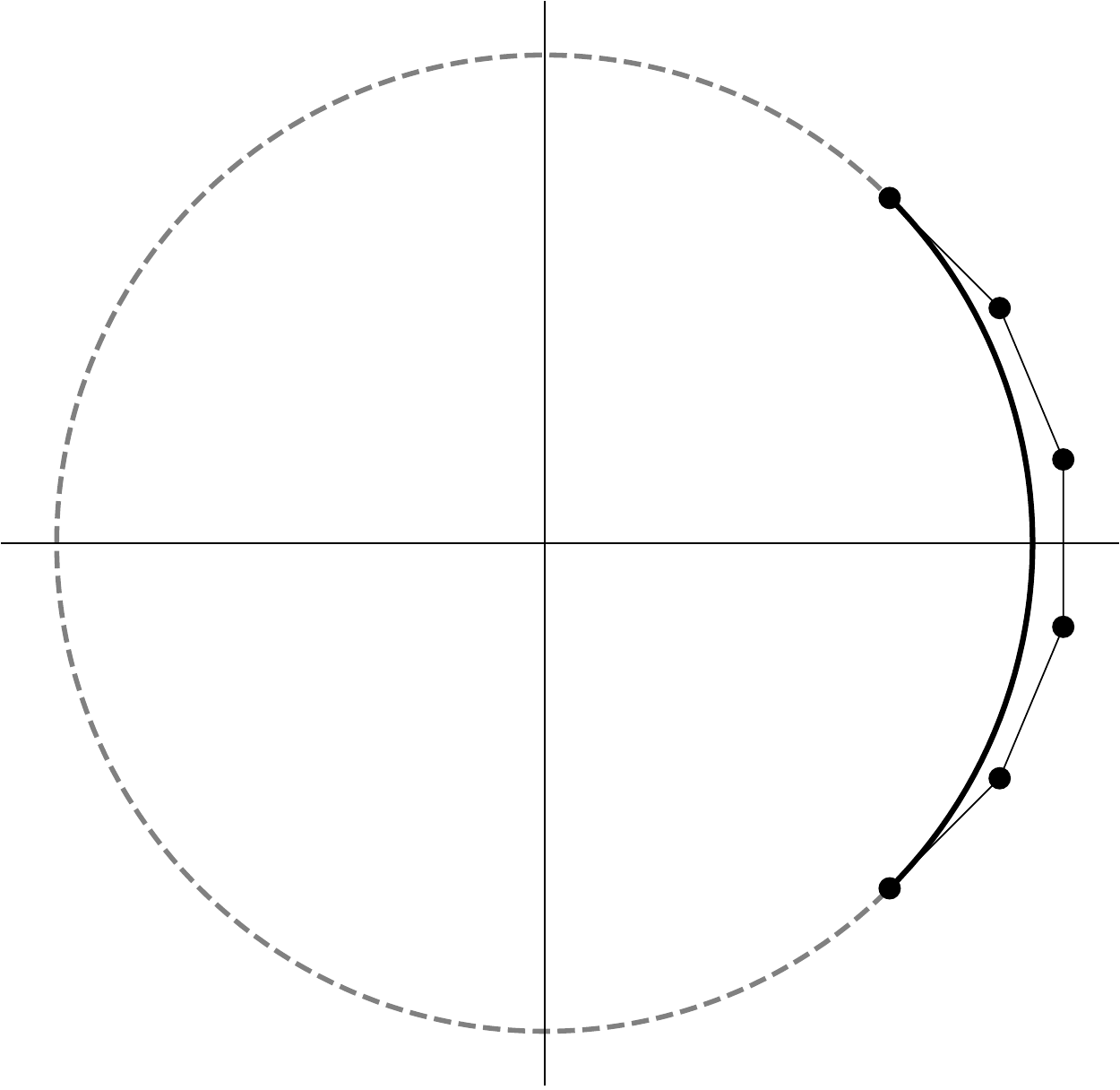}
  \endminipage\hfill
  \minipage{0.33\textwidth}
    \centering
    \includegraphics[width=1.0\linewidth]{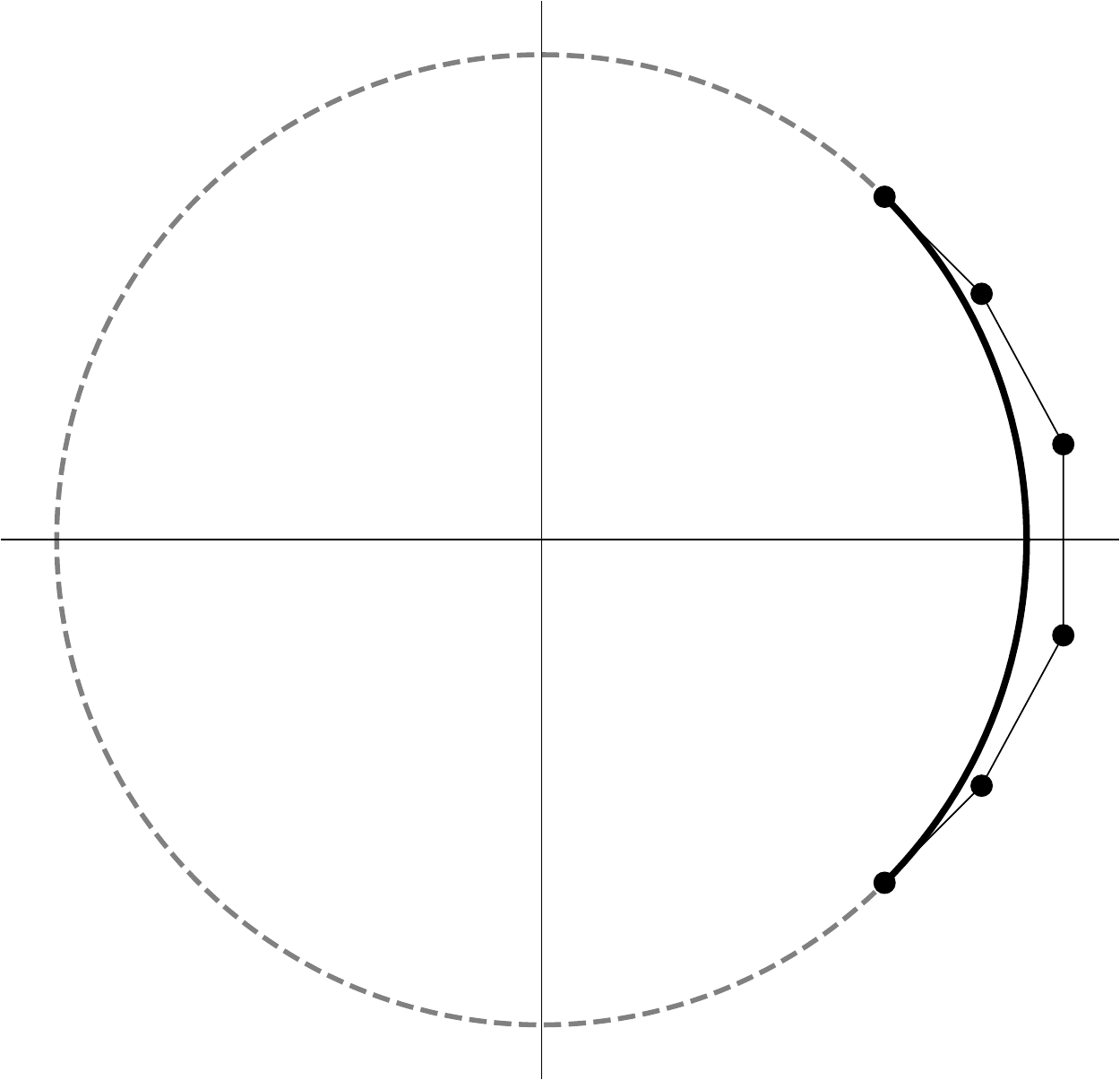}
  \endminipage\hfill
  \minipage{0.33\textwidth}
    \centering
    \includegraphics[width=1.0\linewidth]{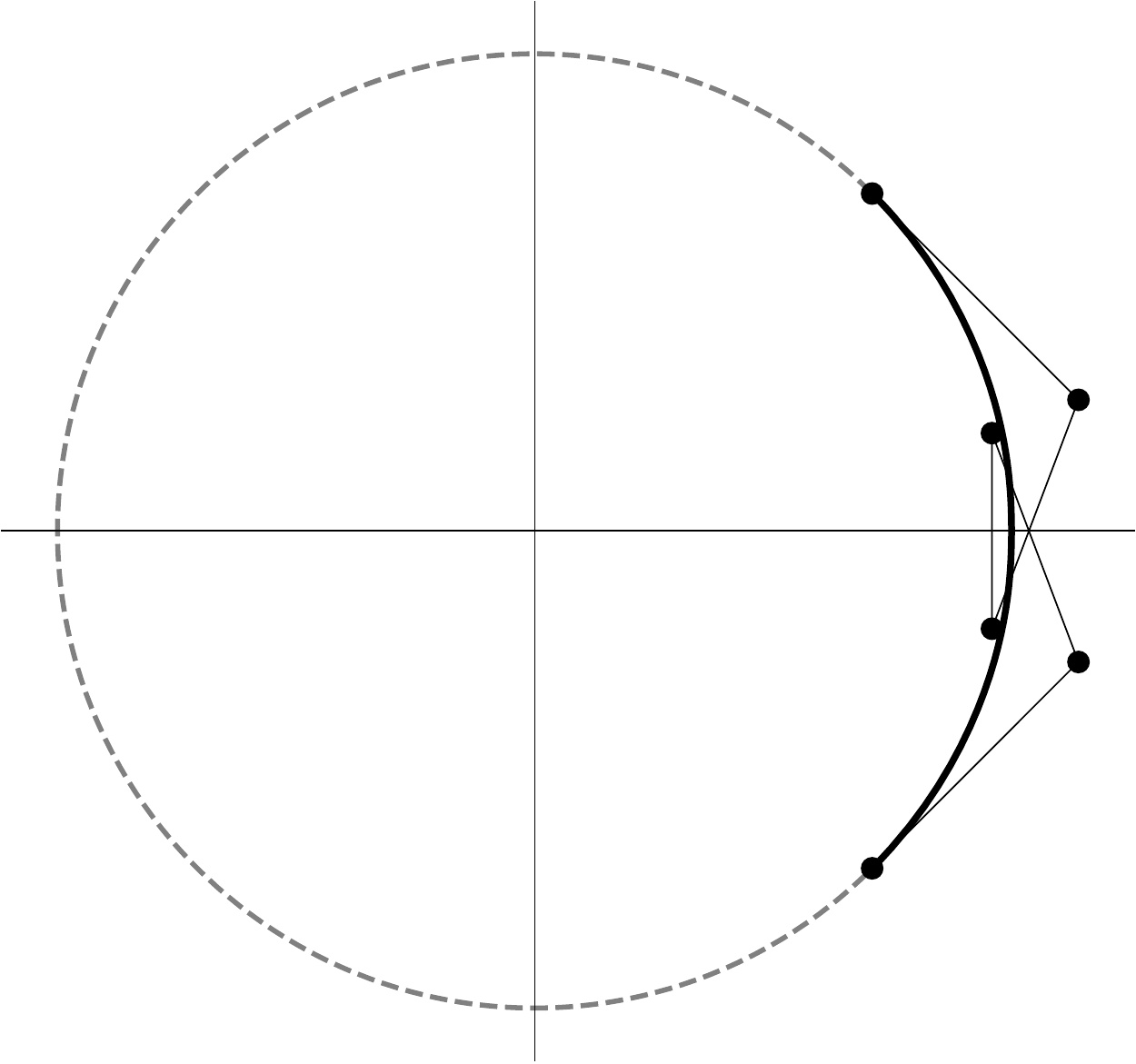}
  \endminipage\hfill

  \caption{Graphs of quintic $G^3$ interpolants (solid black)
  of the circular arc given by the inner angle $\pi/2$ (dashed gray).
  The left one is the best interpolant with the maximum of the radial error
  approximately $2\times 10^{-8}$. The middle one and one on the right
  have an alternating radial error too. But the values of their radial errors
  are approximately $2\times 10^{-6}$ for the middle one,
  and $8\times 10^{-4}$ for the one on the right.}\label{fig:quinticpi4}
\end{figure}

\section{Conclusion}\label{sec:conclusion}

In this paper a new approach to study the circular arc approximation by parametric polynomials
via the geometric interpolation has been considered.
The existence and the uniqueness of the best parametric polynomial
of some particular degree and order of geometric smoothness were confirmed for the simplified
radial error and also for the radial error itself. The later one is a new result which, up to
our knowledge, has not been studied in the literature yet. A general iterative algorithm for
the construction of the best parametric polynomial approximant was established in the case when
the degree of the parametric polynomial $n$ and the order of geometric smoothness $k$ are related
by $k=n-2$. The algorithm is based on a kind of bisection method, thus it is robust and leads to
the desired solution for any starting point of the iteration. Since the cases $n=2,3,4,5$ have been analyzed
only, one can consider the analysis of some higher degrees as a future work. But even more important
is the analysis of some low degree cases for which $n-k>2$. The first step in this direction
was done in \cite{Vavpetic-Zagar-general-circle-19} and in \cite{Vavpetic-CAGD20}
but only for the simplified radial error.
It is clear that the problem of finding the best geometric interpolant in the case of the radial error is
much more challenging issue.


\end{document}